\documentclass[aap,noinfoline]{imsart}
\usepackage{amsmath,amssymb,amsthm,booktabs}
\usepackage{mathrsfs}

\usepackage[utf8]{inputenc}
\usepackage[colorlinks,citecolor=blue,urlcolor=blue]{hyperref}

\usepackage{bm}\usepackage{euscript}
\usepackage{graphicx}

\usepackage{multicol}
\usepackage[usenames,dvipsnames,svgnames,table]{xcolor}

\usepackage[authoryear]{natbib}
\usepackage{a4wide}

\usepackage{mathrsfs}
\usepackage{euscript}

\startlocaldefs
\numberwithin{equation}{section} \theoremstyle{plain}
\newtheorem{theorem}{Theorem}[section]
\newtheorem{lemma}{Lemma}[section]
\newtheorem{corollary}{Corollary}[section]

\newtheorem{remark}{Remark}[section]
\endlocaldefs

\allowdisplaybreaks

\setattribute{journal}{name}{}

\begin{document}

\def\cB{\mathcal{B}}
\def\cD{\mathcal{D}}
\def\cE{\mathcal{E}}
\def\cF{\mathcal{F}}
\def\cH{\mathcal{H}}
\def\cN{\mathcal{N}}
\def\cS{\mathcal{S}}

\def\bC{\mathbb{C}}
\def\bD{\mathbb{D}}
\def\bR{\mathbb{R}}

\newcommand\ES{\EuScript{S}}

\def\blue{} \def\red{}
\begin{frontmatter}

  \title{High-dimensional limits of eigenvalue distributions for general Wishart process}
  
  \runtitle{High-dimensional limits of general Wishart processes}

  \begin{aug}
    \author{\fnms{~ Jian} \snm{Song}\ead[label=e1]{txjsong@hotmail.com}}
    \and
    \author{\fnms{~ Jianfeng} \snm{Yao~}\ead[label=e2]{jeffyao@hku.hk}}
    \and
    \author{\fnms{~ Wangjun} \snm{Yuan}\ead[label=e3]{ywangjun@connect.hku.hk}}
    
    \affiliation{Shandong University and  The University of Hong Kong}
    \runauthor{J. Song,   J. Yao \& W. Yuan}

    \address{School of Mathematics\\
      Shandong University\\
      \printead{e1}}
    
    \address{ 
      Department of Statistics and Actuarial Science\\
      The University of Hong Kong\\
      \printead{e2}
    }

    \address{
      Department of Mathematics \\
      The University of Hong Kong\\
      \printead{e3}
    }
  \end{aug}

  \begin{abstract}

In this article, we obtain an equation for the high-dimensional limit measure of eigenvalues of  generalized Wishart processes, and the results is extended to  random
particle systems  that  generalize SDEs of eigenvalues.  We also introduce a new set of conditions on
the coefficient matrices for the existence and uniqueness of a strong solution for the SDEs of eigenvalues. The equation of the limit measure is further discussed assuming self-similarity on the eigenvalues. 

  \end{abstract}
  


\end{frontmatter}

\section{Introduction}\label{sec:intro}


While the theory of  stochastic differential equations (SDEs)
with values in a Euclidean space is quite well developed 
in stochastic analysis,  the study of  SDEs on general manifolds
is more recent. {{}}{In this paper, we consider the
  eigenvalue process  of the solution of a special class of
  matrix-valued SDEs as well as a more general class of particle systems introduced in \cite{Graczyk2014}.}
For ease of notation, let $\ES_N$ be the group of  $N\times N$  symmetric  matrices.
For $X\in \ES_N$ and $f$ a real-valued function, $f(X)\in\ES_N$
denotes  the matrix obtained from $X$ by acting $f$ on the spectrum of $X$. Namely, if $X$ has the spectral decomposition  $X=\sum_{j=1}^p \alpha_j u_ju_j^\intercal$ with eigenvalues $(\alpha_j)$ and eigenvectors $(u_j)$, then $f(X)=\sum_{j=1}^pf( \alpha_j) u_ju_j^\intercal$.  Here  $A^\intercal$ denotes the transpose of a matrix or  vector $A$.

{There is no much  work }  in the literature on SDEs with matrix state space $\ES_N$. 
We consider the class of so-called {\em generalized
  Wishart process}  which satisfies the following SDE on $\ES_N$
\begin{align} \label{SDE-matrix}
  dX_t^N = g_N(X_t^N) dB_t h_N(X_t^N) + h_N(X_t^N) dB_t^{\intercal}
  g_N(X_t^N) + b_N(X_t^N) dt, \quad t\ge 0.
\end{align}
Here  $B_t$ is a Brownian matrix of dimension $N \times N$, and the
functions $g_N, h_N, b_N : \mathbb{R} \rightarrow \mathbb{R}$ act on the
spectrum of $X_t^N$.
 Let 
\begin{equation}\label{eq-Gn}
G_N(x,y) = g_N^2(x) h_N^2(y) + g_N^2(y) h_N^2(x),
\end{equation} which is
symmetric with respect to $x$ and $y$.  Let  $\lambda_1^N(t) \le
\lambda_2^N(t) \le \cdots \le \lambda_N^N(t)$ be the eigenvalues of $X_t^N$. According to Theorem 3 in \citet{Graczyk2013}, {{}}{if $\lambda_1^N(0) < \lambda_2^N(0) < \cdots < \lambda_N^N(0)$, then} before the {{}}{first} collision time
{{}}{
\begin{align*}
	\tau_N = \inf\{t>0: \exists \ i \neq j, ~\lambda_i(t) = \lambda_j(t) \},
\end{align*}
}
the eigenvalues satisfy the following SDEs: for $1\le i \le N$,
{\small
\begin{align} \label{SDE-eigenvalue}
  d\lambda_i^N(t)
  &= 2g_N(\lambda_i^N(t)) h_N(\lambda_i^N(t)) dW_i(t) + \left( b_N(\lambda_i^N(t)) + \sum_{j:j\neq i} \dfrac{G_N(\lambda_i^N(t), \lambda_j^N(t))}{\lambda_i^N(t) - \lambda_j^N(t)} \right) dt.%
\end{align}}%
Here, $\{W_i, i=1,2,\dots, N\}$ are independent Brownian motions.
{In  \citet{Graczyk2013,Graczyk2014}, some other conditions on the functions
{were imposed to} ensure that \eqref{SDE-eigenvalue}  has a unique strong solution and the collision time is infinity almost surely.}

The generalized Wishart process \eqref{SDE-matrix} extends
the celebrated symmetric Brownian motion and Wishart process introduced respectively in \citet{Dyson62} and \citet{Bru1989}, as follows. 
\begin{itemize}
\item  If we take $g_N(x) = (2N)^{-1/2}$, $h_N(x) = 1$ and $b_N(x) =
  0$ in \eqref{SDE-matrix}, the random matrix $X_t^N$ becomes the
  symmetric Brownian motion with elements:
  \begin{equation}
    \label{eq:Dyson}
    X_t^N(i,j) = \frac1{\sqrt{N}} B_{t}(i,j) {\large\text{1}}_{\{i<j\}} +
    \frac{\sqrt 2}{\sqrt{N}} B_{t}(i,i) {\large\text{1}}_{\{i=j\}},\quad 1\le
    i\le j\le N,
  \end{equation}
  where $\{B_t(i,j), ~i\le j \}$ are independent Brownian motions. 
\item
  If we take 
  $g_N(x) = \sqrt{x}$, $h_N(x) = 1/\sqrt{N}$, and $b_N(x) =p/N$ with
  $p > N-1$ in \eqref{SDE-matrix}, then the random matrix
  $Y_t^N=NX_t^N$ is the Wishart process $B_t^\intercal B_t$, where
  $B_t$ is a $p \times N$ Brownian matrix.
\end{itemize}

Symmetric matrices appear in many scientific fields. Historically,
\citet{Dyson62} used symmetric  Brownian motions to analyse  the Hamiltonian of a complex
nucleic system in particle  physics.
\citet{Bru1989} introduced her Wishart process to perform  principal 
component analysis on a set of  resistance data of Escherichia Coli to
certain antibiotics.
More recently, time series of positive definite  matrices are
particularly important in the following fields. 

\begin{enumerate}
\item Financial data analysis:  multivariate volatility/co-volatility (variance/covariance) between stock returns or
  interest rates from different markets have been studied recently
  through Wishart processes, see \citet{Gourieroux06},
  \citet{Gourieroux10a},  \citet{DaFonseca08},
  \citet{DaFonseca14},   \citet{Gnoatto12}, \citet{Gnoatto14} and \citet{Wu18}.
  
\item Machine learning: an important task in machine learning using
  kernel functions is the determination of  a suitable kernel matrix for a given
  data analysis problem (\citet{Scholkopf02}).  Such determination is
  referred as the  {\em kernel matrix learning problem}.
  A kernel   matrix is in fact a positive definite Gram-matrix of size
  $N\times N$ where  $N$, the sample size of the data, is usually
  large.
  An innovative method for kernel learning is proposed by
  \citet{Zhang06}
  where unknown kernel matrix is modeled by 
  a Wishart process prior. This approach has been followed 
  in \citet{Kondor07} and \citet{Li09}.
  
\item  Computer vision:  real-time computer vision often involves
  tracking of objects of interest. At each time $t$, a target is
  encoded
  into a $N$-dimensional vector $a_t\in\bR^N $ (feature vector).  It is therefore
  clear that measuring ``distance'' between these vectors, say 
  $a_t$ and $a_{t+dt}$  at  two consecutive {time spots} $t$ and
  $t+dt$, is of
  crucial importance for object tracking. Because the standard
  Euclidean distance $\|a_{t+dt} -a_{t}\|^2$ is rarely optimal, it is more
  satisfactory to  identify a better metric of the form
  $ (a_{t+dt} -a_{t})^\intercal M_t  (a_{t+dt}  -a_{t})$ using a
  suitable positive definite matrix $M_t$.
  Again, the sequence of metric matrices  $(M_t)$ is time
  varying; it should be data-adaptive, estimable from data available
  at time $t$. An innovative solution is proposed in \citet{Li16}
  where $M_t$ follows a Wishart process.
\end{enumerate} 

Motivated by these recent applications where the dimension $N$ of 
a matrix process is usually large,  we study in this paper
{\em high-dimensional  limits}
of eigenvalue distributions of the generalized Wishart process
\eqref{SDE-matrix} {as} $N$ tends to infinity. To the best of our
knowledge, such high-dimensional limits are known in the literature
only for
{{}}{some simple cases. An early result is the derivation of
  the Wigner semi-circle law from the eigenvalue empirical measure
  process in \cite{Chan1992} where the symmetric matrix process  has
  independent Ornstein-Uhlenbeck processes as its entries. The results were later  generalized in \cite{Rogers1993} to the following SDEs, 
  \begin{align*}
	dX_j = \sqrt{\dfrac{2\alpha}{N}} dB_j + \left( - \theta X_j + \dfrac{\alpha}{N} \sum_{j:j\neq i} \dfrac{1}{X_i - X_j} \right) dt, \quad 1 \le i \le N, t \ge 0.
  \end{align*}
   \cite{Cepa1997} further generalised  these SDEs to 
   \begin{align*}
     dX_j = \sigma(X_j) dB_j + \left( b(X_j) + \sum_{j:j\neq i} \dfrac{\gamma}{X_i - X_j} \right) dt, \quad 1 \le i \le N, t \ge 0,
\end{align*}
with some coefficient functions $b$, $\sigma$ and constant $\gamma$.
Another important case is the Mar\v{c}enko-Pastur law for 
the eigenvalue empirical measure process derived in 
\cite{Duvillard2001}. The eigenvalues SDEs \eqref{SDE-eigenvalue}
considered in the present  paper generalises the eigenvalue SDEs in
\cite{Chan1992} and \cite{Duvillard2001}, as well as the particle
system in \cite{Rogers1993}. Also  the particle system
\eqref{SDE-particle}
in Section~\ref{sec:particles} which is introduced in \cite{Graczyk2014}
generalizes the particle system in \cite{Cepa1997}.
}

The rest of the paper is organized as follows. In
Section~\ref{sec:eig-val}, we study high-dimensional limits of
eigenvalue distributions of the  generalized Wishart process
\eqref{SDE-matrix}.
In Section~\ref{sec:particles}, our results are extended to a random
particle system  that  generalizes the eigenvalue SDEs~\eqref{SDE-eigenvalue}.
These results from the two sections presuppose that these SDEs have a
unique strong solution (before colliding/exploding time). In
Section~\ref{sec:existence}, we introduce a new set of conditions on
the coefficient matrices in {{}}{\eqref{SDE-eigenvalue} and its generalization, the particle system \eqref{SDE-particle}}  (here the dimension $N$ is fixed). These conditions are thus compared with the ones proposed in
\citet{Graczyk2013,Graczyk2014}. In Section \ref{sec: discussion},
assuming self-similarity on the eigenvalues, we simplify the equation
\eqref{limit point equation 4.3.25} of the limit measure and indicate
its connection with the Hilbert transform operator.

\section{Limit point of empirical measure for eigenvalues}\label{sec:eig-val}

We denote by $M_1(\mathbb{R})$ the set of probability measures on $\mathbb{R}$. Since a probability measure can be viewed as a continuous linear functional on the space $C_b(\mathbb{R})$ of bounded continuous functions,  $M_1(\mathbb{R})$ is a subset of the dual space $C_b(\mathbb{R})^*$ of $C_b(\mathbb R)$. Since the space $C_b(\mathbb{R})$ endowed with the sup norm is a normable space, its dual $C_b(\mathbb{R})^*$ is a Banach space with the dual norm. The space $M_1(\mathbb{R})$  with the norm inherited from the dual norm of $C_b(\mathbb{R})^*$ is complete.  Besides, the space $C([0,T], M_1(\mathbb{R}))$  endowed with the metric
\begin{align*}
	d_{C([0,T], M_1(\mathbb{R}))} (f_1, f_2) = \sup_{t\in[0,T]} d_{M_1(\mathbb{R})} \left( f_1(t), f_2(t) \right),
\end{align*}
is complete.  

 Consider the empirical measure of the eigenvalues $\lambda_i^N(t)$ satisfying \eqref{SDE-eigenvalue}
\begin{align}\label{eq-lnt}
	L_N(t) = \dfrac{1}{N} \sum_{i=1}^N \delta_{\lambda_i^N(t)}.
\end{align}
We shall study the limit point of $L_N$ in the space $C([0,T], M_1(\mathbb{R}))$, as $N$ goes to infinity, and we assume the following conditions. 
\begin{itemize}
\item[(A)] There exists a positive function $\varphi(x) \in C^2(\mathbb{R})$ such that
$\lim\limits_{|x|\rightarrow +\infty} \varphi(x) = +\infty,$
 $\varphi'(x)b_N(x)$ is bounded with respect to $(x,N)$, and $\varphi'(x)g_N(x)h_N(x)$ satisfies
\begin{align*}
	\sum_{N=1}^{\infty} \left( \dfrac{\|\varphi' g_N h_N\|_{L^{\infty}(dx)}^2}{N} \right)^{l_1} < \infty,
\end{align*}
for some positive integer $l_1$. 

\item[(B)]  The function
$N G_N(x,y) \dfrac{\varphi'(x) - \varphi'(y)}{x-y}$
is bounded with respect to $(x,y, N)$. 

 \item[(C)] 
\begin{align}\label{eq-C0}
	C_0 = \sup_{N>0} \langle \varphi, L_N(0) \rangle
	= \sup_{N>0} \dfrac{1}{N} \sum_{i=1}^N \varphi \left( \lambda_i^N(0) \right) < \infty.
\end{align}
\item[(D)] There exists a sequence $\{\tilde{f}_k\}_{k\in \mathbb N}$ of $C^2(\bR)$  functions such that it is dense in the space $C_0(\mathbb{R})$ of continuous functions vanishing at infinity and  that $\tilde{f}_k'(x) g_N(x) h_N(x)$ satisfies
\begin{align}\label{eq-psi}
	\psi(k) = \sum_{N=1}^{\infty} \left( \dfrac{\|\tilde{f}_k' g_N h_N\|_{L^{\infty}(dx)}^2}{N} \right)^{l_2} < \infty
\end{align}
for some positive integer $l_2 \ge 2$.
\end{itemize}

\begin{remark}
When one chooses the function $\varphi(x)$ in condition (A), although
$\varphi(x)$ goes to  $\infty$   as $ |x|$ goes to $\infty$, one
should expect that the first and second derivatives of $\varphi$
vanish fast enough. One typical choice is  $\varphi(x) = \ln (1+x^2)$.

 Condition (B) implies that 
\begin{align*}
	N G_N(x,x) \varphi''(x)
	= \lim_{y \rightarrow x} N G_N(x,y) \dfrac{\varphi'(x) - \varphi'(y)}{x-y}
\end{align*}
is uniformly bounded with respect to $(x,N)$, and so is $Ng_N^2(x)h_N^2(x)\varphi''(x).$
\end{remark}

\begin{remark}\label{rmk2.2}
Suppose that $b_N(x) \le c_b|x|$, $g_N^2(x) \le c_g |x| N^{-\alpha}$ and $h_N^2(x) \le c_h |x| N^{-\beta}$ for large $N$ and large $|x|$ with constants $c_b, c_g, c_h$ and $\alpha + \beta \ge 1$, then we can choose $\varphi(x) = \ln(1 + x^2)$ to satisfy the above conditions (A), (B) and (D).
\end{remark}

\begin{theorem}\label{thm1}
 Let $T>0$ be a fixed number. {Suppose that
   \eqref{SDE-eigenvalue} has a strong solution that is non-exploding
   and non-colliding for  $t \in [0,T]$.} Then under the conditions (A), (B), (C) and (D),  the sequence $\{L_N(t), t\in[0,T]\}_{N\in \mathbb N}$ is relatively compact in $C([0,T], M_1(\mathbb{R}))$, i.e., every subsequence has a further subsequence that converges in $C([0,T], M_1(\mathbb{R}))$ almost surely.
\end{theorem}

\begin{proof}
We split the proof into three steps for the reader's convenience. 

{\bf Step 1.} In this step, we apply It\^o's formula to estimate $\langle f, L_N(t) \rangle$ for $f\in C^2(\mathbb R)$. 

Note that
\begin{align*}
	\langle f, L_N(t) \rangle
	= \int f(x) L_N(t)(dx)
	= \dfrac{1}{N} \sum_{i=1}^N \int f(x) \delta_{\lambda_i^N(t)}(dx)
	= \dfrac{1}{N} \sum_{i=1}^N f(\lambda_i^N(t)).
\end{align*}
By  It\^{o}'s formula and \eqref{SDE-eigenvalue},
\begin{align*}
	f(\lambda_i^N(t))
	&= f(\lambda_i^N(0)) + \int_{0}^t f'(\lambda_i^N(s)) d\lambda_i^N(s) + \dfrac{1}{2} \int_{0}^t f''(\lambda_i^N(s)) d \langle \lambda_i^N \rangle_s \\
	&= f(\lambda_i^N(0)) + 2\int_{0}^t f'(\lambda_i^N(s)) g_N(\lambda_i^N(s)) h_N(\lambda_i^N(s)) dW_i(s) + \int_{0}^t f'(\lambda_i^N(s)) b_N(\lambda_i^N(s)) ds \\
	&\quad + \int_{0}^t f'(\lambda_i^N(s)) \sum_{j:j\neq i} \dfrac{G_N (\lambda_i^N(s), \lambda_j^N(s))}{\lambda_i^N(s) - \lambda_j^N(s)} ds + 2\int_{0}^t f''(\lambda_i^N(s)) g_N^2(\lambda_i^N(s)) h_N^2(\lambda_i^N(s)) ds.
\end{align*}
Thus,
\begin{align} \label{Ito formula}
	\langle f, L_N(t) \rangle
	&= \dfrac{1}{N} \sum_{i=1}^N f(\lambda_i^N(0)) + \dfrac{2}{N} \sum_{i=1}^N \int_{0}^t f'(\lambda_i^N(s)) g_N(\lambda_i^N(s)) h_N(\lambda_i^N(s)) dW_i(s) \nonumber \\
	&\quad + \dfrac{1}{N} \sum_{i=1}^N \int_{0}^t f'(\lambda_i^N(s)) b_N(\lambda_i^N(s)) ds + \dfrac{1}{N} \sum_{i\neq j} \int_{0}^t f'(\lambda_i^N(s)) \dfrac{G_N (\lambda_i^N(s), \lambda_j^N(s))}{\lambda_i^N(s) - \lambda_j^N(s)} ds \nonumber \\
	&\quad + \dfrac{2}{N} \sum_{i=1}^N \int_{0}^t f''(\lambda_i^N(s)) g_N^2(\lambda_i^N(s)) h_N^2(\lambda_i^N(s)) ds\notag\\
	&= \langle f, L_N(0) \rangle + M_f^N(t)+ \int_{0}^t \langle f'b_N, L_N(s) \rangle ds+2\int_{0}^t \langle f''g_N^2h_N^2, L_N(s) \rangle ds\notag\\
	&\quad +\dfrac{1}{N} \sum_{i\neq j} \int_{0}^t f'(\lambda_i^N(s)) \dfrac{G_N (\lambda_i^N(s), \lambda_j^N(s))}{\lambda_i^N(s) - \lambda_j^N(s)} ds, 
\end{align}
where
\begin{align} \label{martingale definition}
	M_f^N(t) = \dfrac{2}{N} \sum_{i=1}^N \int_{0}^t f'(\lambda_i^N(s)) g_N(\lambda_i^N(s)) h_N(\lambda_i^N(s)) dW_i(s)
\end{align}
is a local martingale.

{\red
In the following, we adopt the convention that 	$\frac{f'(x) - f'(y)}{x-y} = f''(x)$ on $\{x=y\}$. We omit the integral domain when it is $\bR$. We also omit the domain of the double integral when it is $\bR^2$.
}

By changing the index in the sum and using the symmetry, the last term in \eqref{Ito formula} can be simplified as follows,
\begin{align*}
	&\quad \dfrac{1}{N} \sum_{i\neq j} \int_{0}^t f'(\lambda_i^N(s)) \dfrac{G_N (\lambda_i^N(s), \lambda_j^N(s))}{\lambda_i^N(s) - \lambda_j^N(s)} ds \\
	&= \dfrac{1}{2N} \sum_{i\neq j} \int_{0}^t \dfrac{f'(\lambda_i^N(s)) - f'(\lambda_j^N(s))}{\lambda_i^N(s) - \lambda_j^N(s)} G_N (\lambda_i^N(s), \lambda_j^N(s)) ds \\
	&= \dfrac{1}{2N} \sum_{i\neq j} \int_{0}^t \iint \dfrac{f'(x) - f'(y)}{x - y} G_N (x, y) \delta_{\lambda_i^N(s)}(dx) \delta_{\lambda_j^N(s)}(dy) ds \\
	&= \dfrac{N}{2} \int_{0}^t \iint \dfrac{f'(x) - f'(y)}{x - y} G_N (x, y) L_N(s)(dx) L_N(s)(dy) ds \\
	&\quad - \dfrac{1}{2N} \sum_{i=1}^N \int_{0}^t \iint \dfrac{f'(x) - f'(y)}{x - y} G_N (x, y) \delta_{\lambda_i^N(s)}(dx) \delta_{\lambda_i^N(s)}(dy) ds,
\end{align*}
Hence, the second term on the right-hand side of the above equation can be simplified as
\begin{align*}
	&\quad \dfrac{1}{2N} \sum_{i=1}^N \int_{0}^t \iint \dfrac{f'(x) - f'(y)}{x - y} G_N (x, y) \delta_{\lambda_i^N(s)}(dx) \delta_{\lambda_i^N(s)}(dy) ds \\
	&= \dfrac{1}{2N} \sum_{i=1}^N \int_{0}^t f''(\lambda_i^N(s)) G_N(\lambda_i^N(s), \lambda_i^N(s)) ds \\
	&= \dfrac{1}{N} \sum_{i=1}^N \int_{0}^t f''(\lambda_i^N(s)) g_N^2(\lambda_i^N(s)) h_N^2(\lambda_i^N(s)) ds \\
	&= \int_{0}^t \langle f''g_N^2h_N^2, L_N(s) \rangle ds.
\end{align*}
Therefore, \eqref{Ito formula} becomes
\begin{align} \label{pair formula 4.3.26}
	\langle f, L_N(t) \rangle
	&= \langle f, L_N(0) \rangle + M_f^N(t) + \int_{0}^t \langle f'b_N, L_N(s) \rangle ds + \int_{0}^t \langle f''g_N^2h_N^2, L_N(s) \rangle ds \nonumber \\
	&\quad + \dfrac{N}{2} \int_{0}^t \iint \dfrac{f'(x) - f'(y)}{x - y} G_N (x, y) L_N(s)(dx) L_N(s)(dy) ds.
\end{align}

Now we assume the boundedness of the following terms,  $\sup\limits_N\left| \langle f, L_N(0) \rangle \right|, \sup\limits_{x,N}|f'(x)b_N(x)| $, \\ $\sup\limits_x|f'(x) g_N(x) h_N(x)|$, $\sup\limits_{x,N}|f''(x) g_N^2(x) h_N^2(x)|$    and	$\sup\limits_{x,y, N}|N G_N(x,y) \frac{f'(x) - f'(y)}{x-y}|$.   Note that  the above assumption is satisfied by the function $\varphi$ appearing in conditions (A), (B) and (C).  

Now the quadratic variation of the local martingale $M_f^N(t)$ has the following estimation
\begin{align} \label{martingale estimation}
	\langle M_f^N \rangle_t
	&= \dfrac{4}{N^2} \sum_{i=1}^N \int_{0}^t \left|f'(\lambda_i^N(s)) g_N(\lambda_i^N(s)) h_N(\lambda_i^N(s))\right|^2 ds \nonumber \\
	&= \dfrac{4}{N} \int_{0}^t \langle |f' g_N h_N|^2, L_N(s) \rangle ds \nonumber \\
		&\le \dfrac{4T}{N} \|f' g_N h_N\|_{L^{\infty}(dx)}^2.
\end{align}
{Thus, $M_f^N(t)$  is a martingale.}

By \eqref{pair formula 4.3.26},  we have
\begin{align}\label{eq7}
	\sup_{t\in[0,T]} |\langle f, L_N(t) \rangle|
	\le \sup_{N>0}|\langle f, L_N(0) \rangle| + \sup_{t\in[0,T]} |M_f^N(t)| + D_0T,
\end{align}
where
\begin{align}\label{eq-D0}
	D_0 = \sup_{N>0} \Bigg\{&\|f' b_N\|_{L^{\infty}(dx)} + \|f'' g_N^2 h_N^2\|_{L^{\infty}(dx)}\notag\\
	 &\quad + \dfrac{1}{2} \left\| N G_N(x,y) \dfrac{f'(x) - f'(y)}{x-y} \right\|_{L^{\infty}(dxdy)} \Bigg\}.
\end{align}

Fix $l \in \mathbb{N}$. By Markov inequality, Burkholder-Davis-Gundy inequality and \eqref{martingale estimation}, there exists a positive constant $\Lambda_l$ depending on $l$ such that for any $\varepsilon>0$,
\begin{align} \label{eq11}
	&\mathbb{P} \left( \sup_{t\in[0,T]} \left| M_f^N(t) \right| \ge \varepsilon \right)
	\le \dfrac{1}{\varepsilon^{2l}} \mathbb{E} \left[ \sup_{t\in[0,T]} \left| M_f^N(t) \right|^{2l} \right]\notag\\
	\le& \dfrac{\Lambda_l}{\varepsilon^{2l}} \mathbb{E} \left[ \langle M_f^N \rangle_T^l \right]
	\le \dfrac{4^l T^l \Lambda_l}{N^l\varepsilon^{2l}} \|f' g_N h_N\|_{L^{\infty}(dx)}^{2l}.
\end{align}
Hence, for  $M > \sup\limits_{N>0}|\langle f, L_N(0) \rangle| + D_0T$, it follows from \eqref{eq7} and \eqref{eq11} that
\begin{align} \label{uniform bounded 4.3.31}
	&\mathbb{P} \left( \sup_{t\in[0,T]} \left| \langle f, L_N(t) \rangle \right| \ge M \right)\notag\\
	\le&~ \mathbb{P} \left( \sup_{t\in[0,T]} \left| M_f^N(t) \right| \ge M - C_0T - \sup_{N>0}|\langle f, L_N(0) \rangle| \right) \nonumber \\
	\le& ~\dfrac{4^l T^l \Lambda_l}{N^l (M - D_0T - \sup\limits_{N>0}|\langle f, L_N(0) \rangle|)^{2l}} \|f' g_N h_N\|_{L^{\infty}(dx)}^{2l}.
\end{align}

{\bf Step 2.} Now we study the H\"older continuity of $\langle f, L_N(t) \rangle$.  

For $t \ge s$,  \eqref{pair formula 4.3.26} implies 
\begin{align*}
	\langle f, L_N(t) \rangle - \langle f, L_N(s) \rangle
	&= M_f^N(t) - M_f^N(s) + \int_s^t \langle f'b_N, L_N(u) \rangle du + \int_s^t \langle f''g_N^2h_N^2, L_N(u) \rangle du \\
	&\quad + \dfrac{N}{2} \int_s^t \iint \dfrac{f'(x) - f'(y)}{x - y} G_N (x, y) L_N(u)(dx) L_N(u)(dy) du.
\end{align*}
Hence,
\begin{align*}
	|\langle f, L_N(t) \rangle - \langle f, L_N(s) \rangle|
	&\le |M_f^N(t) - M_f^N(s)| + (t-s) \|f'b_N\|_{L^{\infty}(dx)} + (t-s) \|f''g_N^2h_N^2\|_{L^{\infty}(dx)} \\
	&\quad + \dfrac{t-s}{2} \left\| N \dfrac{f'(x) - f'(y)}{x - y} G_N (x, y) \right\|_{L^{\infty}(dxdy)} \\
	&\le |M_f^N(t) - M_f^N(s)| + (t-s)D_0,
\end{align*}
where $D_0$ is given in \eqref{eq-D0}. Note that $[0,T]$ can be partitioned into small intervals of length $\eta < D_0^{-8/7}$ and the number of the intervals are $J = [T\eta^{-1}]$. Then by  Markov inequality, Burkholder-Davis-Gundy inequality and \eqref{martingale estimation}, we have
\begin{align*}
	&\mathbb{P} \left( \sup_{|t-s|\le \eta} \left|M_f^N(t) - M_f^N(s) \right| \ge M\eta^{1/8} \right)\\
	\le &\sum_{k=0}^J \mathbb{P} \left( \sup_{k\eta \le t \le (k+1)\eta} \left|M_f^N(t) - M_f^N(k\eta) \right| \ge \dfrac{M\eta^{1/8}}{3} \right) \\
	\le& \sum_{k=0}^J \dfrac{3^{2l}}{M^{2l} \eta^{l/4}} \mathbb{E} \left[ \sup_{k\eta \le t \le (k+1)\eta} \left|M_f^N(t) - M_f^N(k\eta) \right|^{2l} \right] \\
	\le& \sum_{k=0}^J \dfrac{3^{2l} \Lambda_l}{M^{2l} \eta^{l/4}} \mathbb{E} \left[ \langle M_f^N(k\eta+\cdot) - M_f^N(k\eta) \rangle_{\eta}^l \right] \\
	\le& \sum_{k=0}^J \dfrac{6^{2l} \Lambda_l \eta^{3l/4}}{M^{2l} N^l} \|f' g_N h_N\|_{L^{\infty}(dx)}^{2l} \\
\le & \eta^{3l/4-1}\cdot\dfrac{6^{2l} \Lambda_l T }{M^{2l} N^l} \|f' g_N h_N\|_{L^{\infty}(dx)}^{2l}.
\end{align*}
Hence, noting that $\eta D_0<\eta^{1/8},$ we have
\begin{align} \label{equicontinuous estmation 4.3.32}
	&\quad \mathbb{P} \left( \sup_{|t-s|\le \eta} |\langle f, L_N(t) \rangle - \langle f, L_N(s) \rangle| \ge (M+1) \eta^{1/8} \right) \nonumber \\
	&\le \mathbb{P} \left( \sup_{|t-s|\le \eta} \left|M_f^N(t) - M_f^N(s) \right| \ge (M+1) \eta^{1/8} - \eta D_0 \right) \nonumber \\
	&\le \mathbb{P} \left( \sup_{|t-s|\le \eta} \left|M_f^N(t) - M_f^N(s) \right| \ge M \eta^{1/8} \right) \nonumber \\
	&\le \eta^{3l/4-1}\cdot \dfrac{6^{2l} \Lambda_l  T}{M^{2l} N^l} \|f' g_N h_N\|_{L^{\infty}(dx)}^{2l}.
\end{align}

{\bf Step 3.} In this last step, we obtain the relative compactness of  $\{L_N\}_{N \in \mathbb{N}^+}$ and conclude the proof. 

 Let  $M$ denote a generic positive constant that may vary in different places.  Recalling that $\varphi$ is given in condition (A), we set
\begin{align*}
	K(\varphi, M) =	\left\{ \mu \in M_1(\mathbb{R}): \langle \varphi, \mu \rangle = \int \varphi(x) \mu(dx) \le M+1 \right\}.
\end{align*}
Since $\varphi(x)$ is positive and tends to infinity as $|x| \rightarrow +\infty$, $K(\varphi, M)$ is tight, i.e. it is (sequentially) compact in $M_1(\mathbb{R})$.

By Arzela-Ascoli Lemma, the  set 
\begin{align*}
	& C_M(\{\varepsilon_n\}, \{\eta_n\})\\
		= &\bigcap_{n=1}^{\infty} \left\{ g \in C([0,T], \mathbb{R}): \sup_{|t-s|\le \eta_n} |g(t) - g(s)| \le \varepsilon_n, \sup_{t\in[0,T]} |g(t)| \le M \right\},
\end{align*}
where $\{\varepsilon_n\}$ and $\{\eta_n\}$ are two positive sequences converging to $0$, 
 is (sequentially) compact in $C([0,T],\mathbb 
 R)$. For $\varepsilon>0$ and a bounded function $\tilde{f} \in C^2(\mathbb{R})$, we define
\begin{align*}
	C_T(\tilde{f}, \varepsilon)
	&= \bigcap_{n=1}^{\infty} \left\{ \mu \in C([0,T], M_1(\mathbb{R})): \sup_{|t-s|\le n^{-4}} |\mu_t(\tilde{f}) - \mu_s(\tilde{f})| \le \dfrac{1}{\varepsilon \sqrt{n}} \right\} \\
	&= \left\{ \mu \in C([0,T], M_1(\mathbb{R})): \sup_{|t-s|\le n^{-4}} |\mu_t(\tilde{f}) - \mu_s(\tilde{f})| \le \dfrac{1}{\varepsilon \sqrt{n}}, \forall n \in \mathbb{N} \right\} \\
	&= \left\{ \mu \in C([0,T], M_1(\mathbb{R})): t \rightarrow \mu_t(\tilde{f}) \in C_M(\{(\varepsilon\sqrt{n})^{-1}\}, \{n^{-4}\}) \right\},
\end{align*}
where we can choose $M=\|\tilde f\|_\infty$.
 By Lemma 4.3.13 in \cite{Anderson2009}, for a positive sequence $\{\varepsilon_k\}_{k\in \mathbb N}$ which will be determined in the sequel, the set
\begin{align*}
	\mathcal{H}_M = \Bigl\{\mu\in C([0,T], M_1(\mathbb{R})): \mu_t \in K(\varphi, M), \ \forall t \in [0,T]\Bigr\} {{}}{\cap \bigcap_{k=1}^{\infty}} C_T(\tilde{f}_k, \varepsilon_k),
\end{align*}
where $\{\tilde{f}_k\}_{k\ge 1}$ is given in Condition (D), is compact in $C([0,T], M_1(\mathbb{R}))$. We have
\begin{align}
	\sum_{N=1}^{\infty} \mathbb{P} (L_N \in \mathcal{H}_M^\mathsf{c})
	\le &\sum_{N=1}^{\infty} \mathbb{P} (\exists t \in [0,T], \ \mathrm{s.t.} \ L_N(t) \notin K(\varphi, M)) \notag\\
	&\quad + \sum_{N=1}^{\infty} \sum_{k \ge 1} \mathbb{P} (L_N \notin C_T(\tilde{f}_k, \varepsilon_k)).\label{eq-14}
\end{align}
By using \eqref{uniform bounded 4.3.31} for the case $l = l_1$ and $f = \varphi$ with $l_1$ and $\varphi$ given in condition (A), the first term on the right-hand side can be simplified as
\begin{align}
	&\quad \sum_{N=1}^{\infty} \mathbb{P} (\exists t \in [0,T], \ \mathrm{s.t.} \ L_N(t) \notin K(\varphi, M))\notag \\
	&= \sum_{N=1}^{\infty} \mathbb{P} \left( \sup_{t\in[0,T]} \langle \varphi, L_N(t) \rangle > M+1 \right)\notag \\
	&\le \sum_{N=1}^{\infty} \dfrac{4^{l_1} T^{l_1} \Lambda_{l_1}}{N^{l_1} (M + 1 - D_0T - \sup_{N>0}|\langle \varphi, L_N(0) \rangle|)^{2l_1}} \|\varphi' g_N h_N\|_{L^{\infty}(dx)}^{2l_1} \notag\\
	&= \dfrac{4^{l_1} T^{l_1} \Lambda_{l_1}} {(M + 1 - D_0T - C_0)^{2l_1}} \sum_{N=1}^{\infty} \dfrac{\|\varphi' g_N h_N\|_{L^{\infty}(dx)}^{2l_1}}{N^{l_1}} < \infty,\label{eq-15}
\end{align}
where $C_0$ is given by \eqref{eq-C0}, $D_0$ is given by \eqref{eq-D0}, and $M=M_0$ is sufficiently large such that $M_0>D_0T+C_0$.

By using \eqref{equicontinuous estmation 4.3.32} with $l = l_2$, $f = \tilde{f}_k$, $\eta = n^{-4}$ and $M = \varepsilon_k^{-1}-1$, where $l_2$ and $\tilde f_k$ are given in condition (D),  the second term on the right-hand side of \eqref{eq-14} can be simplified as follows, recalling that $\psi(k)$ is given in \eqref{eq-psi},
\begin{align*}
	&\sum_{N=1}^{\infty} \sum_{k \ge 1} \mathbb{P} (L_N \notin C_T(\tilde{f}_k, \varepsilon_k))\\
	\le& \sum_{N=1}^{\infty} \sum_{k \ge 1} \sum_{n=1}^{\infty} \mathbb{P} \left( \sup_{|t-s|\le n^{-4}} |L_N(t)(\tilde{f}_k) - L_N(s)(\tilde{f}_k)| > \dfrac{1}{\varepsilon_k \sqrt{n}} \right) \\
	\le& \sum_{N=1}^{\infty} \sum_{k \ge 1} \sum_{n=1}^{\infty} \dfrac{6^{2l_2} \Lambda_{l_2} T n^{-3l_2+4}}{(\varepsilon_k^{-1}-1)^{2l_2} N^{l_2}} \|\tilde{f}_k' g_N h_N\|_{L^{\infty}(dx)}^{2l_2} \\
	=& 6^{2l_2} \Lambda_{l_2} T \sum_{n=1}^{\infty} n^{-3l_2+4} \sum_{k \ge 1} \dfrac{1}{(\varepsilon_k^{-1}-1)^{2l_2}} \sum_{N=1}^{\infty} \dfrac{\|\tilde{f}_k' g_N h_N\|_{L^{\infty}(dx)}^{2l_2}}{N^{l_2}} \\
	=& 6^{2l_2} \Lambda_{l_2} T \sum_{n=1}^{\infty} n^{-3l_2+4} \sum_{k \ge 1} \dfrac{\psi(k)}{(\varepsilon_k^{-1}-1)^{2l_2}},
\end{align*}
which is finite if we take $\varepsilon_k$ so that $\varepsilon_k^{-1} > 1 + k \psi(k)^{1/(2l_2)}$.

Thus, it follows from \eqref{eq-14}, \eqref{eq-15}, and the above estimate that
\begin{align*}
	\sum_{N=1}^{\infty} \mathbb{P} (L_N \in \mathcal{H}_{M_0}^\mathsf{c}) < \infty,
\end{align*}
and Borel-Cantelli Lemma implies
\begin{align*}
	\mathbb{P} \left( \liminf_{N \rightarrow \infty} \{L_N \in \mathcal{H}_{M_0} \} \right) = 1.
\end{align*}
Finally, the relative compactness of the family $\{L_N\}_{N \in \mathbb{N}^+}$ follows from the compactness of $\mathcal{H}_{M_0}$, and the proof is concluded.
\end{proof}

{\red The Corollary 3 in 
\cite{Graczyk2013} provided the conditions under which the system of SDEs \eqref{SDE-eigenvalue} has a unique non-exploding and non-colliding strong solution. As a consequence, we have the following corollary.

\begin{corollary} \label{Coro-2.1}
For the system of SDEs \eqref{SDE-eigenvalue}, suppose that the initial value satisfies $\lambda_1^N(0) < \cdots < \lambda_N^N(0)$ and the condition (C) holds. Assume that there exist positive constants $L$, $\alpha$ and $\beta$ with $\alpha + \beta \ge 1$, such that $b_N(x)$, $N^{\alpha}g_N^2(x)$ and $N^{\beta}h_N^2(x)$ are Lipschitz continuous with the Lipschitz constant $L$ for all $N \in \mathbb{N}$, and that
\begin{align*}
	\max_{N\in\mathbb{N}}\{|b_N(0)| + N^{\alpha} g_N^2(0) + N^{\beta} h_N^2(0)\} \le L.
\end{align*}
Besides, suppose that $G_N(x,x)$ is convex or in the H\"{o}lder space $\mathcal{C}^{1,1}(\mathbb R)$, and that $G_N(x,y)$ is strictly positive on $\{x \neq y\}$ for all $N \in \mathbb{N}$. Then for any fixed number $T>0$, the sequence $\{L_N(t), t\in[0,T]\}_{N\in \mathbb N}$ is relatively compact in $C([0,T], M_1(\mathbb{R}))$.
\end{corollary}

\begin{proof}
Under the conditions given in the Corollary, by \cite[Corollary 3]{Graczyk2013}, for each $N$, the system of SDEs \eqref{SDE-eigenvalue} has a unique strong solutions that is non-exploding and non-colliding on $[0, \infty)$. Besides, we have the following estimation
\begin{align*}
	|b_N(x)| \le |b_N(0)| + |b_N(x) - b_N(0)| \le L(1 + |x|),
\end{align*}
which is also satisfied by $N^{\alpha} g_N^2(x)$ and $N^{\beta} h_N^2(x)$. Thus, it is easy to check that the conditions (A) (B) and  (D) are now satisfied (with $\varphi(x) = \ln(1+x^2)$), and the conclusion follows from Theorem \ref{thm1}.
\end{proof}
}


Under proper conditions, the following Theorem provides an equation for the Stieltjes transform of the limit point of $\{L_N\}_{N\in \mathbb N}$.

\begin{theorem}\label{thm2}
{{}}{Let $T>0$ be a fixed number. Assume that
  \eqref{SDE-eigenvalue} has a strong solution that is non-exploding
  and non-colliding for  $t \in [0,T]$.} 
Furthermore, we assume that {\red there exist continuous functions $b(x)$ and $G(x,y)$, such that} $b_N(x)$ converges to $b(x)$ and $N G_N(x,y)$ converges to $G(x,y)$ uniformly as $N$ tends to infinity, and that
\begin{align*}
	\left\| \dfrac{b(x)}{1+x^2} \right\|_{L^{\infty}(dx)} < \infty,\
	\left\| \dfrac{G(x,y)}{(1+|x|) (1+y^2)} \right\|_{L^{\infty}(dxdy)} < \infty.
\end{align*}

If {\red almost surely,} the empirical measure $L_N(0)$ converges weakly to a measure $\mu_0$ as $N$ goes to infinity, and  the sequence $\{L_N\}_{N\in \mathbb N}$ has a limit measure $\mu$ in $C([0,T], M_1(\mathbb{R}))$, then the measure $\mu$ satisfies the equation
\begin{align} \label{limit point equation 4.3.25}
	\int \dfrac{\mu_t(dx)}{z-x}
	=& \int \dfrac{\mu_0(dx)}{z-x} + \int_{0}^t \left[ \int \dfrac{b(x)}{(z-x)^2} \mu_s(dx) \right] ds\notag \\
	&+ \int_{0}^t \left[ \iint \dfrac{G(x,y)}{(z-x) (z-y)^2} \mu_s(dx) \mu_s(dy) \right] ds,
\end{align}
for $z \in \mathbb{C} \setminus \mathbb{R}$.
\end{theorem}

\begin{remark}  \label{rmk-2.4} Taking $x=y$,   the boundedness condition
\begin{align*}
	\left\| \dfrac{G(x,y)}{(1+|x|) (1+y^2)} \right\|_{L^{\infty}(dxdy)} < \infty
\end{align*}
becomes
\begin{align*}
	\left\| \dfrac{G(x,x)}{(1+|x|)^3} \right\|_{L^{\infty}(dx)} < \infty.
\end{align*}
Thus,
\begin{align*}
	\dfrac{NG_N(x,x)}{(1+|x|)^4} \le C
\end{align*}
for some constant $C$ and large  $N$. Note that $G_N(x,x) = 2g_N^2(x)h_N^2(x)$, we have
\begin{align*}
	\sum_{N=1}^{\infty} \dfrac{1}{N} \left\| \dfrac{g_N(x) h_N(x)}{(z-x)^2} \right\|_{L^{\infty}(dx)}^2
	= \sum_{N=1}^{\infty} \dfrac{1}{2N} \left\| \dfrac{G_N(x,x)}{(z-x)^4} \right\|_{L^{\infty}(dx)}
	\le \sum_{N=1}^{\infty} \dfrac{C}{2N^2}
	< \infty.
\end{align*}
\end{remark}

\begin{proof} \hskip-0.5mm {(of Theorem~\ref{thm2}.)}\quad 
 For any limit point $\mu = (\mu_t, t \in [0,T])$ of $L_N$, we can find a subsequence $\{N_i\}$, such that $L_{N_i}$ converges to $\mu$ in $C([0,T], M_1(\mathbb{R}))$ as $N_i$ tends to infinity. By using \eqref{pair formula 4.3.26} for the case $N = N_i$ and $f(x) = (z-x)^{-1}$, and then letting $N_i$ tends to infinity, we have
\begin{align} \label{limit of pair}
	\int \dfrac{\mu_t(dx)}{z-x}
	&= \int \dfrac{\mu_0(dx)}{z-x} + \lim_{N_i \rightarrow \infty} M_f^{N_i}(t) + \lim_{N_i \rightarrow \infty} \int_{0}^t \int \dfrac{b_{N_i}(x)}{(z-x)^2} L_{N_i}(s)(dx) ds \nonumber \\
	&\quad + \lim_{N_i \rightarrow \infty} \int_{0}^t \int \dfrac{2g_{N_i}^2(x) h_{N_i}^2(x)}{(z-x)^3} L_{N_i}(s)(dx) ds \nonumber \\
	&\quad + \lim_{N_i \rightarrow \infty} \dfrac{1}{2} \int_{0}^t \iint \dfrac{(z-x)^{-2} - (z-y)^{-2}}{x - y} N_i G_{N_i} (x, y) L_{N_i}(s)(dx) L_{N_i}(s)(dy) ds.
\end{align}

The second term of right-hand side of \eqref{limit of pair} vanishes almost surely. Indeed, by using \eqref{eq11} for the case $l = 1$ and $f(x) = (z-x)^{-1}$ for some $z \in \mathbb{C} \setminus \mathbb{R}$, we have
\begin{align*}
	\sum_{N=1}^{\infty} \mathbb{P} \left( \sup_{t\in[0,T]} \left| M_{f}^N(t) \right| \ge \varepsilon \right)
	\le \sum_{N=1}^{\infty} \dfrac{4T \Lambda_1}{N \varepsilon^2} \left\| \dfrac{g_N(x) h_N(x)}{(z-x)^2} \right\|_{L^{\infty}(dx)}^2,
\end{align*}
of which the right-hand side is finite due to Remark \ref{rmk-2.4}. By Borel-Cantelli Lemma,
\begin{align*}
	\mathbb{P} \left( \liminf_{N \rightarrow \infty} \left\{ \sup_{t\in[0,T]} \left| M_{f}^N(t) \right| < \varepsilon \right\} \right) = 1,
\end{align*}
i.e. $M_{f}^N(t)$ converges to zero uniformly with respect to $t$ almost surely.

For the third term on the right-hand side of \eqref{limit of pair}, noting that the boundedness of $b(x) (1+x^2)^{-1}$ implies the boundedness of $b(x) (z-x)^{-2}$ for $z \in \mathbb{C} \setminus \mathbb{R}$, {\red which is continuous,} we have
\begin{align*}
	&\quad \left| \int \dfrac{b_{N_i}(x)}{(z-x)^2} L_{N_i}(s)(dx) - \int \dfrac{b(x)}{(z-x)^2} \mu_s(dx) \right| \\
	&\le \left| \int \dfrac{b_{N_i}(x) - b(x)}{(z-x)^2} L_{N_i}(s)(dx) \right| + \left| \int \dfrac{b(x)}{(z-x)^2} (L_{N_i}(s)(dx) - \mu_s(dx)) \right| \\
	&\le \dfrac{\sup_{x} |b_{N_i}(x) - b(x)|}{(\mathrm{Im}(z))^2} + \left| \int \dfrac{b(x)}{(z-x)^2} (L_{N_i}(s)(dx) - \mu_s(dx)) \right|,
\end{align*}
the right-hand of which converges to 0 as $N_i\to\infty$ {\red by the uniform convergence of $b_{N_i}(x)$ towards $b(x)$ and the weak convergence of the empirical measure $L_{N_i}(s)$ towards $\mu_s$}. Besides, the boundedness of $b(x)/(1 + x^2)$ and the uniform convergence of $b_N(x)$ to $b(x)$ imply the boundedness of $b_N(x)/(z-x)^2$. Then it follows from the dominated convergence theorem that
\begin{align*}
	\lim_{N_i \rightarrow \infty} \int_{0}^t \int \dfrac{b_{N_i}(x)}{(z-x)^2} L_{N_i}(s)(dx) ds
	= \int_{0}^t  \int \dfrac{b(x)}{(z-x)^2} \mu_s(dx)  ds.
\end{align*}

Similarly, for the fourth term on the right-hand side of \eqref{limit of pair}, noting that $2N_i g_{N_i}^2(x) h_{N_i}^2(x) = N_i G_{N_i}(x,x)$, we have
\begin{align*}
	&\quad \left| \int \dfrac{2g_{N_i}^2(x) h_{N_i}^2(x)}{(z-x)^3} L_{N_i}(s)(dx) \right| = \dfrac{1}{N_i} \left| \int \dfrac{N_i G_{N_i}(x,x)}{(z-x)^3} L_{N_i}(s)(dx) \right| \\
	&\le \dfrac{1}{N_i} \left| \int \dfrac{N_i G_{N_i}(x,x) - G(x,x)}{(z-x)^3} L_{N_i}(s)(dx) \right| + \dfrac{1}{N_i} \left| \int \dfrac{G(x,x)}{(z-x)^3} L_{N_i}(s)(dx) \right| \\
	&\le \dfrac{\sup_{x,y} |N_i G_{N_i}(x,x) - G(x,x)|}{N_i (\mathrm{Im}(z))^3} + \dfrac{C_z}{N_i} \left\| \dfrac{G(x,x)}{(1+|x|^3)} \right\|_{L^{\infty}(dx)} \end{align*}
	which tend to 0 as $N_i\to\infty$. Here, $C_z$ is a constant depending only on $z$.

Finally, using the identity
\begin{align*}
	&\dfrac{(z-x)^{-2} - (z-y)^{-2}}{x - y}
	= \dfrac{(z-y)^2 - (z-x)^2}{(z-x)^2 (z-y)^2 (x-y)} \\
	=& \dfrac{2z-x-y}{(z-x)^2 (z-y)^2} 
	= \dfrac{1}{(z-x) (z-y)^2} + \dfrac{1}{(z-x)^2 (z-y)},
\end{align*}
the last term on the right-hand side of \eqref{limit of pair} can be simplified as
\begin{align*}
	&\quad \lim_{N_i \rightarrow \infty} \dfrac{1}{2} \int_{0}^t \iint \dfrac{(z-x)^{-2} - (z-y)^{-2}}{x - y} N_i G_{N_i} (x, y) L_{N_i}(s)(dx) L_{N_i}(s)(dy) ds \\
	&= \lim_{N_i \rightarrow \infty} \dfrac{1}{2} \int_{0}^t \iint \left[ \dfrac{1}{(z-x) (z-y)^2} + \dfrac{1}{(z-x)^2 (z-y)} \right] N_i G_{N_i} (x, y) L_{N_i}(s)(dx) L_{N_i}(s)(dy) ds \\
	&= \lim_{N_i \rightarrow \infty} \int_{0}^t \iint \dfrac{N_i G_{N_i} (x, y)}{(z-x) (z-y)^2} L_{N_i}(s)(dx) L_{N_i}(s)(dy) ds,
\end{align*}
where the last equality follows from the symmetry of $G_{N_i}$. Now,
\begin{align*}
	&\quad \left| \iint \dfrac{N_i G_{N_i} (x, y)}{(z-x) (z-y)^2} L_{N_i}(s)(dx) L_{N_i}(s)(dy) - \iint \dfrac{G(x, y)}{(z-x) (z-y)^2} \mu_s(dx) \mu_s(dy) \right| \\
	&\le \left| \iint \dfrac{N_i G_{N_i} (x, y) - G(x,y)}{(z-x) (z-y)^2} L_{N_i}(s)(dx) L_{N_i}(s)(dy)\right| \\
	&\quad + \left| \iint \dfrac{G(x, y)}{(z-x) (z-y)^2} (L_{N_i}(s)(dx) L_{N_i}(s)(dy) - \mu_s(dx) \mu_s(dy)) \right| \\
	&\le  \dfrac{\sup_{x,y} |N_i G_{N_i} (x, y) - G(x,y)|}{|\mathrm{Im}(z)|^3} + \left| \iint \dfrac{G(x, y)}{(z-x) (z-y)^2} (L_{N_i}(s)(dx) L_{N_i}(s)(dy) - \mu_s(dx) \mu_s(dy)) \right| 
\end{align*}
converges to 0 as $N_i\to \infty$. {{}}{Also note that the boundedness of $G(x,y)/(1 + |x|)(1 + y^2)$ and the uniform convergence of $NG_N(x,y)$ to $G(x,y)$ yield the boundedness of $NG_N(x,y)/(z-x)(z-y)^2$.} Thus, by the dominated convergence theorem {\red and the continuity of the function $G(x,y)$},
\begin{align*}
	\lim_{N_i \rightarrow \infty} \int_{0}^t \iint \dfrac{N_i G_{N_i} (x, y)}{(z-x) (z-y)^2} L_{N_i}(s)(dx) L_{N_i}(s)(dy) ds = \int_{0}^t \iint \dfrac{G(x, y)}{(z-x) (z-y)^2} \mu_s(dx) \mu_s(dy) ds.
\end{align*}
Therefore, \eqref{limit point equation 4.3.25}  is obtained from \eqref{limit of pair}.  The proof is complete. 
\end{proof}

{\red
Using the conditions in Corollary 3 of \cite{Graczyk2013} that guarantee the existence and uniqueness of the non-exploding and non-colliding strong solution to the system of SDEs \eqref{SDE-eigenvalue}, we have the following corollary.
	
\begin{corollary} \label{Coro-2.2}
For the system of SDEs \eqref{SDE-eigenvalue}, suppose $\lambda_1^N(0) < \cdots < \lambda_N^N(0)$. Assume that there exist positive constants $L$ and $\alpha$, such that $b_N(x), N^{\alpha} g_N^2(x), N^{1-\alpha}h_N^2(x)$ are Lipschitz continuous with the Lipschitz constant $L$ for all $N \in \mathbb{N}$, and
\begin{align*}
	\max_{N\in\mathbb{N}}\{|b_N(0)| + N^{\alpha} g_N^2(0) + N^{1-\alpha} h_N^2(0)\} \le L.
\end{align*}
Besides, suppose that $G_N(x,x)$ is convex or in the H\"{o}lder space $\mathcal{C}^{1,1}$, and that $G_N(x,y)$ is strictly positive on $\{x \neq y\}$. Moreover, assume that $b_N(x)$ converges to a continuous function $b(x)$ and $NG_N(x,y)$ converges to a continuous function $G(x,y)$ uniformly as $N$ tends to infinity.
	
If the empirical measure $L_N(0)$ converges weakly to a measure $\mu_0$ almost surely   as $N$ goes to infinity, and  the sequence $\{L_N\}_{N\in \mathbb N}$ has a limit measure $\mu$ in $C([0,T], M_1(\mathbb{R}))$ for any fixed number $T > 0$, then the measure $\mu$ satisfies the equation
	\begin{align} \label{eq-2.17'}
		\int \dfrac{\mu_t(dx)}{z-x}
		=& \int \dfrac{\mu_0(dx)}{z-x} + \int_{0}^t \left[ \int \dfrac{b(x)}{(z-x)^2} \mu_s(dx) \right] ds\notag \\
		&+ \int_{0}^t \left[ \iint \dfrac{G(x,y)}{(z-x) (z-y)^2} \mu_s(dx) \mu_s(dy) \right] ds,
	\end{align}
	for $z \in \mathbb{C} \setminus \mathbb{R}$, $t \in [0,T]$.
\end{corollary}

\begin{proof}
By \cite[Corollary 3]{Graczyk2013}, we can conclude that for each $N$, the SDEs \eqref{SDE-eigenvalue} has a unique strong solution that is non-exploding and non-colliding on $[0, \infty)$. Moreover, the estimation in the proof of Corollary \ref{Coro-2.1} is still valid for $b_N(x)$, $N^{\alpha} g_N^2(x)$ and $N^{1-\alpha} h_N^2(x)$. Besides, we have
\begin{align*}
	|NG_N(x,y)|
	&\le \left( |N^{\alpha} g_N^2(x) - N^{\alpha} g_N^2(0)| + |N^{\alpha} g_N^2(0)| \right) \left( |N^{1-\alpha} h_N^2(y) - N^{1-\alpha} h_N^2(0)| + |N^{1-\alpha} h_N^2(0)| \right) \\
	&\le L^2 (1+|x|)(1+|y|).
\end{align*}
It can be easily checked that all the conditions in Theorem \ref{thm2} are satisfied.
\end{proof}

}

\begin{remark}[The normalized case]
Now we suppose that $Y_t^N$ satisfies the following equation
\begin{align} \label{non-normalized SDE for matrix}
	dY_t^N = g(Y_t^N) dB_t h(Y_t^N) + h(Y_t^N) dB_t^{\intercal} g(Y_t^N) + a(Y_t^N) dt.
\end{align}
Then the  equation for $X_t^N := \dfrac{1}{N} Y_t^N$ is
\begin{align*}
	dX_t^N = \dfrac{1}{N} g(NX_t^N) dB_t h(NX_t^N) + \dfrac{1}{N} h(NX_t^N) dB_t^{\intercal} g(NX_t^N) + \dfrac{1}{N} a(NX_t^N) dt,
\end{align*}
which coincides with \eqref{SDE-matrix} with
\begin{align*}
	g_N(x) h_N(y) = \dfrac{1}{N} g(Nx) h(Ny) \text{ and }
	b_N(x) = \dfrac{1}{N} a(Nx).
\end{align*}
Therefore, under the conditions in Theorem \ref{thm1} and Theorem \ref{thm2}, the equation \eqref{limit point equation 4.3.25} is still valid for a limit measure $\mu$ of the empirical measures of the eigenvalues of $X^N$ with
\begin{align*}
	b(x) = \lim_{N \rightarrow \infty} \dfrac{1}{N} a(Nx) \text{ and }
	G(x,y) = \lim_{N \rightarrow \infty} \dfrac{1}{N} [g^2(Nx) h^2(Ny) + h^2(Nx) g^2(Ny)].
\end{align*}
\end{remark}

\section{Limit point of empirical measure for particle systems}
\label{sec:particles}

In \cite{Graczyk2014}, the following system of SDEs was introduced:
for  $1 \le i \le N$ and  $ t \ge 0$,
\begin{align} \label{SDE-particle}
	dx_i^N(t) = \sigma_i^N(x_i^N(t)) dW_i(t)+ \left( b_N(x_i^N(t)) + \sum_{j:j\neq i} \dfrac{H_N(x_i^N(t), x_j^N(t))}{x_i^N(t) - x_j^N(t)} \right) dt,
\end{align}
where $H_N(x,y)$ is a non-negative symmetric function, and the existence and
uniqueness of the non-colliding strong solution was studied. Clearly,
this particle system generalizes the system \eqref{SDE-eigenvalue} for
eigenvalues of a generalized Wishart process studied in
Section~\ref{sec:eig-val}.
{\red There is a huge literature on related interacting particle systems,
{particularly} on those related to Bessel processes. For background
  information,  we here  refer to
  the survey papers \citet{SurveyBessel03} and
  \citet{Zambotti17}, and  the recent book \citet{Katori16}.
}

In this section, we extend the results established in Section~\ref{sec:eig-val} for the particle system. Here the corresponding empirical measures are
\begin{align*}
	L_N(t) = \dfrac{1}{N} \sum_{i=1}^N \delta_{x_i^N(t)}
\end{align*}
We assume the following conditions which are similar to those in Section~\ref{sec:eig-val}.

\begin{itemize}
\item[(A')] There exists a positive function $\varphi(x) \in C^2(\mathbb{R})$ such that
$\lim\limits_{|x|\rightarrow +\infty} \varphi(x) = +\infty,$
 $\varphi'(x)b_N(x)$ is bounded with respect to $(x,N)$, and $\varphi''(x) \sigma_i^N(x)^2$ is bounded with respect to $(x, i, N)$,  
 and $\varphi'(x) \sigma_i^N(x)$ satisfies
\begin{align*}
	\sum_{N=1}^{\infty} \left( \dfrac{\max\limits_{1 \le i \le N} \|\varphi'  \sigma_i^N\|_{L^{\infty}(dx)}^2}{N} \right)^{l_1} < \infty
\end{align*}
for some positive integer $l_1$. 

\item[(B')] The function
$	N H_N(x,y) \dfrac{\varphi'(x) - \varphi'(y)}{x-y}$
is bounded with respect to $(x,y, N)$. 

 \item[(C')] 
\begin{align*}
	 C_0':=\sup_{N>0} \langle \varphi, L_N(0) \rangle
	= \sup_{N>0} \dfrac{1}{N} \sum_{i=1}^N \varphi \left( \lambda_i^N(0) \right) < \infty.
\end{align*}
\item[(D')]  There exists a sequence  $\{\tilde{f}_k\}_{k\ge 1}$ of
  $C^2(\bR)$ functions such that it is dense in $C_0(\mathbb{R})$ and that $\tilde{f}_k'(x) \sigma_i^N(x)$ satisfies
\begin{align*}
	\psi(k) = \sum_{N=1}^{\infty} \left( \dfrac{\max\limits_{1 \le i \le N} \|\tilde{f}_k' \sigma_i^N\|_{L^{\infty}(dx)}^2}{N} \right)^{l_2} < \infty
\end{align*}
for some positive integer $l_2 \ge 2$.
\end{itemize}

\begin{remark} Similar to Remark \ref{rmk2.2},
suppose that $b_N(x) \le c_b|x|$, $\sigma_i^N(x) \le c_0 |x|$ and $H_N(x,y) \le c_h |xy| N^{-\gamma}$ for large $N$ and large $|x|, |y|$ with constants $c_b, c_0, c_h$ and $\gamma \ge 1$, then we can choose $\varphi(x) = \ln(1 + x^2)$ to satisfy the conditions.
\end{remark}

\begin{theorem}\label{thm3}
 Let $T>0$ be a fixed number. {Suppose that \eqref{SDE-particle} has a strong solution that is non-exploding and non-colliding for $t \in [0,T]$.} Then under the conditions (A'), (B'), (C') and (D'),  the sequence $\{L_N(t), t\in[0,T]\}_{N\in\mathbb N}$ is relatively compact in $C([0,T], M_1(\mathbb{R}))$ almost surely.
\end{theorem}

{\red
Analogous to Corollary \ref{Coro-2.1}, we have the following corollary of Theorem \ref{thm3},  based on the conditions given in \cite{Graczyk2014} which assures a unique non-exploding and non-colliding strong solution to \eqref{SDE-particle}.

\begin{corollary} \label{Coro-3.1}
For the system of SDEs \eqref{SDE-particle}, assume that the initial value satisfies $\lambda_1^N(0) < \cdots < \lambda_N^N(0)$ and condition (C'). Suppose that $\sigma_i^N(x)^2$ is Lipschitz continuous for each $1 \le i \le N$, and $H_N(x,y)$ is continuous for all $N\in \mathbb{N}$. Moreover, there exists a positive number $L>0$, such that $b_N(x)$ is Lipschitz continuous with the Lipschitz constant $L$ for all $N \in \mathbb{N}$, and
\begin{align*}
	\sup_{N\in\mathbb{N}}\{|b_N(0)|\} \le L.
\end{align*}
Besides, suppose that there exist constant $c_2 \ge 0$ that does not depend on $N$, and constants $c_3(N), c_4(N)$ that may depend  on $N$, such that for $1 \le i \le N$,
\begin{enumerate}
	\item[(a)] $H_N(x,y) \le \dfrac{c_2}{N} (1 + |xy|)$, $\forall x, y \in \bR$;
	\item[(b)] $\frac{H_N(w,z)}{z-w} \le \dfrac{H_N(x,y)}{y-x}$, $\forall w < x < y < z$;
	\item[(c)] $\sigma_i^N(x)^2 + \sigma_i^N(y)^2 \le c_3(N) (x-y)^2 + 4H_N(x,y)$, $\forall x, y \in \bR$;
	\item[(d)] $H_N(x,y)(y-x) + H_N(y,z)(z-y) \le c_4(N) (z-y)(z-x)(y-x) + H_N(x,z)(z-x)$, $\forall x < y < z$.
\end{enumerate}
Then for any fixed number $T > 0$, the sequence $\{L_N(t), t\in[0,T]\}_{N\in\mathbb N}$ is relatively compact in $C([0,T], M_1(\mathbb{R}))$ almost surely.
\end{corollary}

\begin{proof}
As in the proof of Corollary \ref{Coro-2.1},  the estimation $|b_N(x)| \le L(1+|x|)$ holds. By (a) and (c), we have $\sigma_i^N(x)^2 \le 2H_N(x,x) \le 2c_2 (1 + |x|^2)/N$. By \cite{Graczyk2014},  the system \eqref{SDE-particle} has a unique strong solution that is non-exploding and non-colliding on $[0, \infty)$,  for each $N\in \mathbb N$. Besides, it can be easily checked that the conditions (A') (B') and (D') are satisfied with $\varphi(x) = \ln (1+x^2)$. Thus, the desired result  comes from Theorem \ref{thm3}.
\end{proof}

}

A similar equation for the Stieltjes transform of the limit measure is given below. 

\begin{theorem}\label{thm4}
{Let $T>0$ be a fixed number. Assume that \eqref{SDE-particle} has a strong solution that is non-exploding and non-colliding for $t \in [0,T]$.} Suppose that
\begin{align*}
	\sum_{N=1}^{\infty} \left( \dfrac{1}{N} \max_{1 \le i \le N} \left\| \dfrac{\sigma_i^N(x)}{1+x^2} \right\|_{L^{\infty}(dx)}^2 \right)^{l_3} < \infty
\end{align*}
for some positive integer $l_3$, and  that there exists a {\red continuous} function $\sigma(x)$ such that
\begin{align}\label{eq-sigma}
	\lim_{N \rightarrow \infty} \max_{1 \le i \le N} \left\| \dfrac{\sigma_i^N(x)^2 - \sigma(x)^2}{1+x^3} \right\|_{L^{\infty}(dx)} = 0.
\end{align}

Furthermore, assume that {\red there exist continuous functions $b(x)$ and $H(x,y)$, such that} $b_N(x)$ converges to $b(x)$ and $N H_N(x,y)$ converges to $H(x,y)$ uniformly as $N$ tends to infinity, and that
\begin{align*}
	\left\| \dfrac{b(x)}{1+x^2} \right\|_{L^{\infty}(dx)} < \infty,\
	\left\| \dfrac{H(x,y)}{(1+|x|) (1+y^2)} \right\|_{L^{\infty}(dxdy)} < \infty, \ \left\| \dfrac{\sigma(x)^2}{1+x^3} \right\|_{L^{\infty}(dx)} < \infty.
\end{align*}

If the empirical measure $L_N(0)$ converges weakly as $N$ goes to infinity to a measure $\mu_0$ almost surely, and  the sequence $L_N$ has a limit measure $\mu$ in $C([0,T], M_1(\mathbb{R}))$, then the measure $\mu$ satisfies the equation
\begin{align} 
	\int \dfrac{\mu_t(dx)}{z-x}
	=& \int \dfrac{\mu_0(dx)}{z-x} + \int_{0}^t \left[ \int \dfrac{b(x)}{(z-x)^2} \mu_s(dx) \right] ds +\int_{0}^t \left[ \int \dfrac{\sigma(x)^2}{(z-x)^3} \mu_s(dx) \right] ds\notag \\
	&+ \int_{0}^t \left[ \iint \dfrac{H(x,y)}{(z-x) (z-y)^2} \mu_s(dx) \mu_s(dy) \right] ds,\label{eq-19}
\end{align}
for $z \in \mathbb{C} \setminus \mathbb{R}$.
\end{theorem}

{\red Similar to Corollary \ref{Coro-3.1}, we have the following consequence of Theorem \ref{thm4}.}

{\red
\begin{corollary} \label{Coro-3.2}
Assume that the initial value of \eqref{SDE-particle} satisfies $\lambda_1^N(0) < \cdots < \lambda_N^N(0)$. Suppose that $\sigma_i^N(x)^2$ is Lipschitz continuous for all $1 \le i \le N$ and $H_N(x,y)$ is continuous for all $N \in \mathbb{N}$. Moreover, assume that there exists a positive number $L > 0$, such that , $b_N(x)$ is Lipschitz continuous with the Lipschitz constant $L$ for all $N \in \mathbb{N}$, and
\begin{align*}
	\sup_{N\in\mathbb{N}} \{|b_N(0)|\} \le L.
\end{align*}
Besides, suppose that the conditions (a) - (d) in Corollary \ref{Coro-3.1} hold. Furthermore, assume that there exist continuous functions $b(x)$ and $H(x,y)$, such that $b_N(x)$ converges to $b(x)$ and $N H_N(x,y)$ converges to $H(x,y)$ uniformly as $N$ tends to infinity.

If almost surely, the empirical measure $L_N(0)$ converges weakly as $N$ goes to infinity to a measure $\mu_0$, and the sequence $L_N$ has a limit measure $\mu$ in $C([0,T], M_1(\mathbb{R}))$ for a fixed number $T > 0$, then the measure $\mu$ satisfies the equation
\begin{align} 
	\int \dfrac{\mu_t(dx)}{z-x}
	= \int \dfrac{\mu_0(dx)}{z-x} + \int_{0}^t \left[ \int \dfrac{b(x)}{(z-x)^2} \mu_s(dx) \right] ds + \int_{0}^t \left[ \iint \dfrac{H(x,y)}{(z-x) (z-y)^2} \mu_s(dx) \mu_s(dy) \right] ds,\label{eq-19'}
\end{align}
for $z \in \mathbb{C} \setminus \mathbb{R}$, $t \in [0,T]$.
\end{corollary}

\begin{proof}
By the proof of Corollary \ref{Coro-3.1}, we have the following estimation
\begin{align*}
	|b_N(x)| \le L(1+|x|), \quad \sigma_i^N(x)^2 \le \dfrac{2c_2}{N} (1+|x|^2).
\end{align*}
Hence, according to \cite{Graczyk2014}, for each $N$, the system \eqref{SDE-particle} has a  unique strong solution that is non-exploding and non-colliding on $[0, \infty)$. It can be checked easily that all the conditions in Theorem \ref{thm4} holds with $\sigma(x) = 0$.
\end{proof}

}

The proofs of Theorems~\ref{thm3} and \ref{thm4}
{\red are analogous to} those of Theorems~\ref{thm1} and \ref{thm2} in
Section~\ref{sec:eig-val}, respectively. They are thus omitted.

\begin{remark}[The normalized case] \label{normalized}

For the  particle system
\begin{align} \label{non-normalized SDE for general particle system}
	dy_i^N(t) = \sigma_i(y_i^N(t)) dW_i(t) + \left( a(y_i^N(t)) + \sum_{j:j\neq i} \dfrac{G(y_i^N(t), y_j^N(t))}{y_i^N(t) - y_j^N(t)} \right) dt, \ 1 \le i \le N, \ t \ge 0,
\end{align}
where $G(x,y)$ is a symmetric function, the normalized particle system
\begin{align*}
	x_i^N(t) = \dfrac{1}{N} y_i^N(t), \ 1 \le i \le N, \ t \ge 0,
\end{align*}
satisfies \eqref{SDE-particle} with
\begin{align*}
	\sigma_i^N (x) = \dfrac{1}{N} \sigma_i(Nx),\
	b_N(x) = \dfrac{1}{N} a(Nx), \text{ and }
	H_N(x,y) = \dfrac{1}{N^2} G(Nx,Ny).
\end{align*}
In this case, if the conditions in Theorem \ref{thm3} and Theorem \ref{thm4} hold,  any limit point $\mu$ of the empirical measures of $\{x_i^N,1\le i\le N\}$ satisfies \eqref{eq-19} with 
\begin{align*}
	\sigma(x)^2 = \lim_{N \rightarrow \infty} \sigma_i^N(x)^2, \
	b(x) = \lim_{N \rightarrow \infty} \dfrac{1}{N} a(Nx), \text{ and }
	H(x,y) = \lim_{N \rightarrow \infty} \dfrac{1}{N} G(Nx,Ny).
\end{align*}
\end{remark}
\medskip

In the rest of this section, we apply the above general results to
{\red general non-colliding squared Bessel particle, general non-colliding squared $\beta$-Bessel particle system,} and Dyson  Brownian motion. \\

{\bf {General non-colliding squared Bessel particle
    system.}} ~{We choose the coefficient functions $g_N(x)$, $h_N(x)$ and $b_N(x)$ {\red and the initial value} in \eqref{SDE-eigenvalue} such that they satisfy the conditions in {\red Corollary \ref{Coro-2.1} and \ref{Coro-2.2}, where $N G_N(x,y) = N (g_N(x)^2 h_N(y)^2 + g_N(y)^2 h_N(x)^2)$ converges to $G(x,y) = x+y$, and $b_N(x)$ converges to $b(x) = c$, uniformly as $N$ tends to infinity.}}
Thus the equation \eqref{eq-2.17'} for the limit measure becomes
\begin{align}
	\int \dfrac{\mu_t(dx)}{z-x}
	= &\int \dfrac{\mu_0(dx)}{z-x} + \int_{0}^t \left[ \int \dfrac{c}{(z-x)^2} \mu_s(dx) \right] ds\notag\\
	 &+ \int_{0}^t \left[ \iint \dfrac{x+y}{(z-x) (z-y)^2} \mu_s(dx) \mu_s(dy) \right] ds.\label{eq-22}
\end{align}
{\red
However, it is challenging to determine the limit measure $\{\mu_t,t \in [0,T]\}$ in general. {If we assume that $\mu_0(dx) = \delta_0(dx)$ and  that}  $\mu_t$ is supported on $[0,\infty)$ for all $t \ge 0$, then \eqref{eq-22} has {a unique} solution as established in \cite{Duvillard2001}. The paper also determined the solution by iterating the equation of the associated characteristic function, {\red for which Gronwall's lemma was employed to deduce the  convergence}.
}


{\red Here we sketch an alternative approach to  find this particular $\{\mu_t,t \in [0,T]\}$. Actually, $\mu_t$ can be considered as the limit of empirical measure of the eigenvalues of $X_t^N=\frac1NB_t^\intercal B_t$ where $B_t$ is a $p\times N$ Brownian matrix. Note that $X_t^N$ and its eigenvalues  solve \eqref{SDE-matrix} and \eqref{SDE-eigenvalue}, respectively, with $g_N(x) = \sqrt{x}/\sqrt{N}$, $h_N(x) = 1$ and $b_N(x) = p/N$. Here, $p > N-1$ and $p/N \rightarrow c \ge 1$.}

{
Denoting {\red the Stieltjes transform of  $\mu_t$ by}
\begin{align}\label{eq-Gt}
	G_t(z) = \int \dfrac{1}{z-x} \mu_t(dx),
\end{align}
 the equation \eqref{eq-22} becomes
\begin{align} \label{Wishart limit equation}
	G_t(z) = G_0(z) - (c-1) \int_{0}^t \partial_z G_s(z) ds - \int_{0}^t \left(G_s(z)^2 + 2z G_s(z) \partial_z G_s(z)\right) ds.
\end{align}
Assume $X_0^N = 0$, and the key observation to solve \eqref{Wishart limit equation} is the following scaling property
\begin{align} \label{scaling property for Wishart}
	G_t(z) = \dfrac{1}{t} G_1 \left( \dfrac{z}{t} \right),
\end{align}
which follows easily from the self-similarity of the process $(B_t^\intercal B_t)_{t\ge0}$.
By \eqref{scaling property for Wishart}, we have 
\begin{align*}
	\partial_z G_t(z)
	= \dfrac{1}{t^2} G_1' \left( \dfrac{z}{t} \right)
	= - \dfrac{1}{z} \dfrac{d}{dt} \left( G_1 \left( \dfrac{z}{t} \right) \right),
\end{align*}
and
\begin{align*}
	G_t(z)^2 + 2z G_t(z) \partial_z G_t(z)
	= \dfrac{1}{t^2} G_1^2 \left( \dfrac{z}{t} \right) + \dfrac{2z}{t^3} G_1 \left( \dfrac{z}{t} \right) G_1' \left( \dfrac{z}{t} \right)
	= - \dfrac{d}{dt} \left( \dfrac{1}{t} G_1^2 \left( \dfrac{z}{t} \right) \right).
\end{align*}
The above two equations and \eqref{Wishart limit equation} imply 
\begin{align}  \label{equation for Wishart stieltjes transform}
	G_t(z)
	&= G_0(z) + \dfrac{c-1}{z} G_1 \left( \dfrac{z}{t} \right) + \dfrac{1}{t} G_1^2 \left( \dfrac{z}{t} \right).
\end{align}
Let $t=1$ in \eqref{equation for Wishart stieltjes transform} and we have
\begin{align} \label{equation of G_1 Wishart}
	z G_1^2(z) + (c-1-z) G_1(z) + 1 = 0,
\end{align}
 of which the solution is
\begin{align} \label{solution of G_1 Wishart}
	G_1(z) = \dfrac{(z+1-c) - \sqrt{(c-1-z)^2 - 4z}}{2z},
\end{align}
where the square root maps from $\mathbb{C}_{+}$ to $\mathbb{C}_{+}$.  Thus by \eqref{scaling property for Wishart},
}
\begin{align} \label{Wishart Stieltjes transform}
	G_t(z) = \dfrac{(z+t(1-c)) - \sqrt{(z+t(1-c))^2 - 4tz}}{2tz}.
\end{align}



%
\begin{remark}
{\red  The matrix process}
\begin{align*}
	\tilde{X}^N(t) = \dfrac{1}{p} B_t^\intercal B_t = \dfrac{N}{p} X^N(t)
\end{align*}
{\red often appears in the literature. We take the notation
$	\tilde{c} = \lim_{N \to \infty} \frac{N}{p} = \frac{1}{c} \le 1. $
 Let  $\tilde{\mu}_t$ be the limit of the empirical measure of $\tilde{X}^N(t)$, and 
denote its Stieltjes transform by
\begin{align*}
	\tilde G_{t}(z) = \int \dfrac{1}{x-z} \tilde{\mu}_t(dx).
	\end{align*}}
Noting that $\tilde{X}^N(t)$ and $X^N(t)$ only differ by a multiple of $N/p$, we also have
$	\tilde{\lambda}_i^N(t) = \frac{N}{p} \lambda_i^N(t)$ {\red and  it is easy to verify that}
\begin{align*}
	\tilde G_{t}(z)=-c G_t(cz).
\end{align*}
Letting $t=1$, we have by \eqref{solution of G_1 Wishart},
\begin{align*}
	\tilde G_1(z)=-cG_1(cz)= \frac{1-\tilde c-z+\sqrt{(1-\tilde c-z)^2-4{\tilde c}^2z}}{2\tilde c z}
\end{align*}
which {\red is the Stieltjes transform of the} \emph{standard}  Mar\v{c}enko-Pastur law with parameter
$\tilde{c} \le 1$ (see, e.g., Equation (3.1.1) in \cite{BSbook}).
\end{remark}

{\bf  {General non-colliding squared $\beta$-Bessel particle system}.} ~
This process is a slight generalization of non-colliding squared Bessel particle system. {We choose the coefficient functions $\sigma_i^N(x), b_N(x), H_N(x,y)$ in \eqref{SDE-particle} such that {\red they satisfy the conditions in Corollary \ref{Coro-3.1} and Corollary \ref{Coro-3.2}, where $b_N(x)$ converges to $b(x) = \beta c$, and $NH_N(x,y)$ converges to $H(x,y) = \beta (x+y)$, uniformly as $N$ tends to infinity, and $\sigma(x) = 0$.} Then the equation \eqref{eq-19'} now is
\begin{align*}
	\int \dfrac{\mu_t(dx)}{z-x}
	&= \int \dfrac{\mu_0(dx)}{z-x} + \beta \int_{0}^t \left[ \int \dfrac{c}{(z-x)^2} \mu_s(dx) \right] ds \nonumber \\
	&\quad + \beta \int_{0}^t \left[ \iint \dfrac{x+y}{(z-x) (z-y)^2} \mu_s(dx) \mu_s(dy) \right] ds,
\end{align*}
which is equivalent to
\begin{align} \label{beta Wishart limit point equation}
	G_t(z) = G_0(z) - \beta (c-1) \int_{0}^t \partial_z G_s(z) ds - \beta \int_{0}^t \left(G_s(z)^2 + 2z G_s(z) \partial_z G_s(z)\right) ds,
\end{align}
where $G_t(z)$ is the Stieltjes transform defined in \eqref{eq-Gt}.


{\red Similar to general non-colliding squared Bessel particle system case, we consider the system of SDEs \eqref{SDE-particle} with $\mu_0(dx) = \delta_0(x)$, $\sigma_i^N(x) = \sqrt{x}/\sqrt{N}$, $H_N(x,y) = \beta (x+y)/N$ and $b_N(x) = b_N$, where $\{b_N, N\in\mathbb N\}$ is a sequence of positive numbers that converges to $\beta c$. By the uniqueness of the solution to \eqref{SDE-particle} and the self-similarity of Brownian motion, we can still obtain the scaling property \eqref{scaling property for Wishart} for $G_t(z)$. Thus, similar to the transformation from \eqref{Wishart limit equation} into \eqref{equation for Wishart stieltjes transform}, \eqref{beta Wishart limit point equation} now is transformed into
\begin{align*}
	G_t(z) = G_0(z) + \dfrac{\beta (c-1)}{z} G_1 \left( \dfrac{z}{t} \right) + \dfrac{\beta}{t} G_1^2 \left( \dfrac{z}{t} \right).
\end{align*}
Letting $t=1$, it is easy to get
\begin{align*}
	G_1(z) = \dfrac{z - \beta(c-1) - \sqrt{[\beta(c-1) - z]^2 - 4\beta z}}{2\beta z}.
\end{align*}
Hence, by \eqref{scaling property for Wishart} ,
\begin{align} \label{beta-Wishart Stieltjes transform}
	G_t(z) = \dfrac{z - \beta t(c-1) - \sqrt{[\beta t(c-1) - z]^2 - 4\beta t z}}{2\beta t z}.
\end{align}
In other words, $\mu_t$ is the celebrated Mar\v{c}enko-Pastur
  law with parameters $(1/c, c\beta t)$.}

\begin{remark}
If we take $\sigma_i(x) = 2 \sqrt{x}$, $a(x) = \beta \alpha$, $G(x,y) = \beta (x+y)$ in \eqref{non-normalized SDE for general particle system} with $\alpha/N \rightarrow c$, the equation becomes
\begin{align} \label{general beta Wishart SDE}
	dy_i^N(t) = 2 \sqrt{y_i^N(t)} dW_i(t) + \beta \left( \alpha + \sum_{j:j\neq i} \dfrac{y_i^N(t) + y_j^N(t)}{y_i^N(t) - y_j^N(t)} \right) dt.
\end{align}
This is the eigenvalue process of the classical $\beta$-Laguerre processes that are studied in \cite{Demni2007} and \cite{Konig2001}. As discussed in Remark \ref{normalized}, the corresponding normalized particle equation is \eqref{SDE-particle} with coefficient functions $\sigma_i^N(x) = 2 \sqrt{x/N}$, $b_N(x) = \beta \alpha / N$ and $H_N(x,y) = \beta (x+y) / N$.
\end{remark}

{\bf {General} Dyson  Brownian motion.}
~{{}}{We choose the coefficient functions $g_N(x)$, $h_N(x)$ and $b_N(x)$ and initial value in \eqref{SDE-eigenvalue} such that they satisfy the conditions in {\red Corollary \ref{Coro-2.1} and Corollary \ref{Coro-2.2}, where $N G_N(x,y) = N (g_N(x)^2 h_N(y)^2 + g_N(y)^2 h_N(x)^2)$ converges to $G(x,y) = 1$, and $b_N(x)$ converges to $b(x) = 0$, uniformly as $N$ tends to infinity.}}

Similar to the {examples above}, \eqref{eq-2.17'} can be simplified as
\begin{align} \label{limit point equation for Dyson}
	G_t(z) = G_0(z) - \int_{0}^t G_s(z) \partial_z G_s(z) ds,
\end{align}
which was shown in \cite{Anderson2009}.

{\red
Now we consider the system of SDEs \eqref{SDE-eigenvalue} with $\mu_0(dx) = \delta_0(dx)$, $g_N(x) = (2N)^{-1/2}$, $h_N(x) = 1$ and $b_N(x) = b_N$, where $\{b_N, N \in \mathbb{N}\}$ is a sequence of positive numbers that converges to $0$. {Thanks to the} uniqueness of the solution to \eqref{SDE-eigenvalue} and the self-similarity of Brownian motion, we can obtain the following scaling property
}

\begin{align}\label{eq-32}
	G_t(z) = \dfrac{1}{\sqrt{t}} G_1 \left( \dfrac{z}{\sqrt{t}} \right).
\end{align}
Thus, \eqref{limit point equation for Dyson} can be transformed to
\begin{align*}
	G_t(z) = G_0(z) + \dfrac{1}{z} G_1^2 \left( \dfrac{z}{\sqrt{t}} \right).
\end{align*}
 When $t=1$,  we have 
\begin{align*}
	G_1(z) = \dfrac{z - \sqrt{z^2-4}}{2},
\end{align*}
which is the Stieltjes transform of the semicircle law. Finally, it follows from the scaling property \eqref{eq-32}  that
\begin{align} \label{solution for G_t Dyson}
	G_t(z) = \dfrac{z - \sqrt{z^2-4t}}{2t},
\end{align}
is the Stieltjes transform of a limit measure, which is also a solution to \eqref{limit point equation for Dyson}.  This  yields the
uniqueness of the limit measure of $L_N$. Note that in \cite{Anderson2009},  the uniqueness of the limit measure was obtained from the uniqueness of the solution to the equation \eqref{limit point equation for Dyson}.

{
\begin{remark}
The symmetric Brownian motion is obtained by taking $g_N(x) = (2N)^{-1/2}$, $h_N(x) = 1$ and $b_N(x) = 0$ in \eqref{SDE-matrix} and the solution of the corresponding eigenvalue SDEs \eqref{SDE-eigenvalue} is the classical Dyson Brownian motion. 
\end{remark}
}

\section{Conditions for existence and uniqueness of the solutions to particle systems} \label{sec:existence} 		  

We stress that the results of large-$N$ limit  in Sections~\ref{sec:eig-val} and \ref{sec:particles} were obtained  under the assumption that the eigenvalue SDEs \eqref{SDE-eigenvalue} and \eqref{SDE-particle} have solutions
(before colliding/exploding).  { Also note that \cite{Graczyk2013,Graczyk2014} imposed  conditions to guarantee the existence and uniqueness of such solutions.}

 In this section, we {{}}{provide a new set of} conditions
 for the existence {{}}{and uniqueness} of
 {{}}{strong} solutions to \eqref{SDE-eigenvalue} and
 \eqref{SDE-particle}. {{}}{Throughout this section, the
   dimension $N$ is fixed and we remove $N$ in subscripts/superscritps.} 

As \eqref{SDE-eigenvalue} is a special case of
\eqref{SDE-particle}, we consider the latter only: for $1\le i\le N$ and $t\ge 0$, 
\begin{equation} \label{SDE-particle-G}
\begin{cases}  &dx_i = \sigma_i(x_i) dW_i(t) + \left( b_i(x_i) + \sum\limits_{j:j\neq i} \dfrac{H_{ij}(x_i, x_j)}{x_i - x_j} \right) dt, \\\
  &x_1(0)<\cdots<x_N(0),
  \end{cases}
\end{equation}
where $(W_i)_{1 \le i \le N}$ are independent Brownian motions. 
In \cite{Graczyk2014},  the existence and strong uniqueness of the
system~\eqref{SDE-particle-G} were established under the following conditions:
\begin{itemize}
\item[(G1)] The functions $\sigma_i$ are  continuous. Besides, there exists a function $\rho : \mathbb{R}_+ \rightarrow \mathbb{R}_+$, such that for any $\varepsilon>0$
  \[
	 \int_0^\varepsilon \rho^{-1}(x) dx = \infty,
  \]
  and  for all $x,y \in \mathbb{R}$ and $ 1 \le i \le N$,
  \[
    |\sigma_i(x) - \sigma_i(y)|^2 \le \rho(|x-y|).
  \]
  
\item[(G2)] The functions $b_i$ and $H_{ij}$ are continuous for all $1
  \le i,j \le N$ and $i \neq j$. {{}}{The functions $H_{ij}$ are non-negative and symmetric, i.e. $H_{ij}(x,y) = H_{ji}(y,x)$}

\end{itemize}

{{}}{Now, we define, for $n\in\mathbb N$, $-\infty \le A < B \le +\infty$},
\begin{align*}
  D^n = \left\{ (x_1, \cdots, x_N) : -\infty<A_n<x_1<\cdots<x_N<B_n<\infty, ~x_{i+1} - x_i > \dfrac{1}{n} \text{ for } 1 \le i \le N-1 \right\},
\end{align*}
{{}}{with $A_n \searrow A$, $B_n \nearrow B$} and define
\begin{align*}
  D = \{ (x_1, \cdots, x_N) : {{}}{A <} x_1 < \cdots < x_N {{}}{< B} \}.
\end{align*}
Then $\overline{D^n} \subseteq D^{n+1}$ and $\bigcup_n D^n = D$.
We impose the following  conditions on the coefficient functions:

\begin{itemize}
\item[(E)] The functions $\sigma_i$ are in $ C^1({(A,B)})$
  and {\red strictly positive on $(A,B)$};
\item[(F)] For each $n \in \mathbb{N}$,  there exists  a number $p
  = p(n) > N$ {\red such that the functions  $b_i(x)$ are in $L^p(A_n, B_n)$ for $1\le i\le N$ and
  $H_{jk}(x,y)$ belongs to $ L^p(\{(x,y|A_n < x < y < B_n, y-x \ge \dfrac{1}{n})\})$ for $1 \le j < k \le N$.}
\end{itemize}

Note that  condition (G1) is not a consequence of condition (E) (consider, e.g.,  $\sigma_i(x) = x^2+1$), and condition (G2) clearly implies condition (F).

\begin{theorem}\label{thm1b}
{{}}{Suppose that the initial value $(x_1(0), \ldots, x_N(0)) \in D$.} Under the conditions (E) and (F), the system of SDEs  \eqref{SDE-particle-G} has a unique strong solution up to the first exit time {{}}{$\tau$ from} D, {{}}{which is defined as follows
\begin{align*}
	\tau = \inf_{t \ge 0}\Bigg\{ (x_1(t), \ldots, x_N(t)) \notin D \Bigg\}.
\end{align*}
}
\end{theorem}

The proof of Theorem~\ref{thm1b} relies on the following result due to  \citet{Krylov2005}.
\begin{theorem}\label{lem:Krylov} 
    Consider  the SDE 
  \begin{align} \label{singular drift SDE}
	x_t = x_0 + \int_0^t b(s+r, x_r) dr + w_t, \quad  t \ge 0,
  \end{align}
  where $w_t$ is a Brownian motion and $b(t,x)$  a $\mathbb{R}^d$-valued Borel function on an open
  set $Q \subseteq \mathbb{R} \times \mathbb{R}^d$.   Let $Q^n, ~n \ge 1$ be bounded open subsets of $Q$, such that $\overline{Q^n} \subseteq Q^{n+1}$ and $\bigcup_n Q^n = Q$. Suppose that for each $n \in \mathbb{N}^+$, there exist $p = p(n) \ge 2$ and $q = q(n) > 2$ satisfying
  \begin{align*}
    \dfrac{d}{p} + \dfrac{2}{q} < 1,
  \end{align*}
  and
  \[
    \left\|   \| b(t,x) I_{Q^n}(t,x)\|_{ L^p(dx)} \right\|_{   L^q(dt)} <  \infty .
  \]
  Then there exists a unique strong solution up to the first exit time,
  say $\tau$, 
  from $Q$. Moreover this solution satisfies 
  \begin{align*}
	\int_0^t |b(s+r, x_r)|^2 dr < \infty
  \end{align*}
  for $t <\tau$ almost surely.
\end{theorem}

\begin{proof}{(of Theorem~\ref{thm1b}.)}\quad 
  By condition (E), for $1 \le i \le N$, there exist $f_i(x) \in C^2({{}}{(A,B)})$ satisfying
  $	f_i'(x) = {1}/{\sigma_i(x)}.$
  Besides, $f_i(x)$ is increasing so it is invertible and the inverse is in $C^2({{}}{(f_i(A),f_i(B))})$.
  For $1 \le i \le N$, let $y_i = f_i(x_i)$.  By It\^{o} formula,
  \begin{align} \label{SDE change variable}
	dy_i
	&= f_i'(x_i) dx_i + \dfrac{1}{2} f_i''(x_i) d\langle x_i \rangle \nonumber \\
	&= f_i'(x_i) \sigma_i(x_i) dW_i + f_i'(x_i) \left( b_i(x_i) + \sum_{j:j\neq i} \dfrac{H_{ij}(x_i, x_j)}{x_i - x_j} \right) dt + \dfrac{1}{2} f_i''(x_i) \sigma_i(x_i)^2 dt \nonumber \\
	&= dW_i + \dfrac{1}{\sigma_i(x_i)} \left( b_i(x_i) + \sum_{j:j\neq i} \dfrac{H_{ij}(x_i, x_j)}{x_i - x_j} \right) dt - \dfrac{1}{2} ({\sigma_i}(x_i))' dt \nonumber \\
	&= dW_i + \dfrac{1}{\sigma_i(f_i^{-1}(y_i))} \left( b_i(f_i^{-1}(y_i)) + \sum_{j:j\neq i} \dfrac{H_{ij}(f_i^{-1}(y_i), f_j^{-1}(y_j))}{f_i^{-1}(y_i) - f_j^{-1}(y_j)} \right) dt \nonumber \\
	&\qquad - \dfrac{1}{2} (\sigma_i(f_i^{-1}(y_i)))' dt.
  \end{align}
  Introduce  the map
  \begin{align*}
	F : \quad {{}}{(A,B)}^N \quad &\longrightarrow {{}}{(f_1(A),f_1(B)) \times \cdots \times (f_N(A),f_N(B))}, \\
	(x_1, \cdots, x_N) &\longmapsto \qquad (f_1(x_1), \cdots, f_N(x_N)).
  \end{align*}
  Then $F$ is bijective, both $F$ and $F^{-1}$ being {\red twice continuously differentiable}.	 Then the system of SDEs \eqref{SDE change variable} on $F(D)$ is equivalent to the the system of SDEs \eqref{SDE-particle-G} on $D$.

  Let  $Q = \mathbb{R}_+ \times F(D)$ and $Q^n = (0,n) \times
  F(D^n)$. In order to apply Theorem~\ref{lem:Krylov},  we only need to verify that the following functions are in $L^p(Q^n)$ for some $p = p(n) > N$:
\begin{align*}
	\dfrac{b_i(f_i^{-1}(y_i))}{\sigma_i(f_i^{-1}(y_i))}, \
	\dfrac{1}{\sigma_i(f_i^{-1}(y_i))} \dfrac{H_{ij}(f_i^{-1}(y_i), f_j^{-1}(y_j))}{f_i^{-1}(y_i) - f_j^{-1}(y_j)}, \text{ and } 
	(\sigma_i(f_i^{-1}(y_i)))'.
\end{align*}
 By change of variables, it is equivalent to show that the functions
\begin{align*}
	\left( \dfrac{b_i(x_i)}{\sigma_i(x_i)} \right)^p \dfrac{1}{\sigma_i(x_i)},\
	\left( \dfrac{1}{\sigma_i(x_i)} \dfrac{H_{ij}(x_i, x_j)}{x_i - x_j} \right)^p \dfrac{1}{\sigma_i(x_i) \sigma_j(x_j)},\text{ and }
	\dfrac{((\sigma_i(x_i))')^p}{\sigma_i(x_i)} 
\end{align*}
belong to $L^1(D^n)$, which is a direct consequence of Conditions (E) and (F). 

The proof is concluded. 
\end{proof}

\begin{remark} \label{rmk-4.1}
{Note that theorem \ref{thm1b} is valid for  Dyson Brownian motion, non-colliding square Bessel process and non-colliding squared $\beta$-Bessel particle system. Indeed, for the Dyson Brownian motion, $\sigma_i(x) = (2N)^{-1/2}$, $b_i(x) = 0$ and $H_{ij}(x,y) = 1/N$, which satisfy the conditions (E) and (F) with $A = - \infty$ and $B = + \infty$. For the non-colliding square Bessel process, $\sigma_i(x) = 2 \sqrt{x}/\sqrt{N}$, $b_i(x) = p/N$ and $H_{ij}(x,y) = (x+y)/N$, which satisfy the conditions (E) and (F) with $A = 0$ and $B = + \infty$. For the non-colliding squared $\beta$-Bessel particle system, $\sigma_i(x) = 2 \sqrt{x/N}$, $b_i(x) = \beta \alpha / N$ and $H_{ij}(x,y) = \beta (x+y) / N$, which also satisfy the conditions (E) and (F) with $A = 0$ and $B = + \infty$. In the non-colliding square Bessel process case and the non-colliding squared $\beta$-Bessel particle system case, the first exit time $\tau$ is the first time the particles explode, collide or reach zero.}

Furthermore, Theorem \ref{thm1b}  also applies to  {\red the particle system \eqref{SDE-particle-G} with discontinuous coefficient functions $b_i(x)$ and $H_{i,j}(x,y)$. For instance, it applies to the system  with  $ \sigma_i(x) = (2N)^{-1/2}$, $b_i(x) = \frac1N f(x)$ and $H_{ij}(x,y) = \frac1Ng(x,y)$ where $f$ and $g$ are bounded measurable functions.
}
\end{remark}

{\red
Combining Theorem \ref{thm1b} with Theorem \ref{thm3} and Theorem \ref{thm4} which are obtained in Section \ref{sec:particles}, we have the following two corollaries for the particle system \eqref{SDE-particle}, in which now the continuity of the coefficient functions $b_N(x)$ and $H_N(x,y)$ is not required.

\begin{corollary} \label{Coro-4.1}
For the system of SDEs \eqref{SDE-particle}, assume that the initial value satisfies $\lambda_1^N(0) < \cdots < \lambda_N^N(0)$ and condition (C') holds. Suppose that for each $N \in \mathbb{N}$, $\sigma_i^N(x)$ are in $C^1(\bR)$ and strictly positive for $1 \le i \le N$ and $b_N(x)$ is non-decreasing (or Lipschitz continuous). Moreover, we assume that there exist positive constants $c_1$, $c_2$ that does not depend on $N$ and positive constants $c_3(N)$ and $c_4(N)$, such that
\begin{enumerate}
	\item [(a')] $|b_N(x)| \le c_1 \sqrt{1+|x|^2}$, $\forall x \in \bR$;
	\item [(b')] $H_N(x,y) \le \frac{c_2}{N} (1+|xy|)$, $\forall x,y \in \bR$;		\item[(c')] $\sigma_i^N(x)^2 + \sigma_i^N(y)^2 \le c_3(N) (x-y)^2 + 4H_N(x,y)$, $\forall x, y \in \bR$, $\forall x,y \in \bR$;
	\item[(d')] $H_N(x,y)(y-x) + H_N(y,z)(z-y) \le c_4(N) (z-y)(z-x)(y-x) + H_N(x,z)(z-x)$, $\forall x < y < z$.
\end{enumerate}

Then for any fixed number $T > 0$, the sequence $\{L_N(t), t\in[0,T]\}_{N\in\mathbb N}$ is relatively compact in $C([0,T], M_1(\mathbb{R}))$ almost surely.
\end{corollary}

\begin{proof}
It is obvious that conditions (a') and (b') imply condition (F). Thus, by Theorem \eqref{thm1b},  SDEs \eqref{SDE-particle} has a unique strong solution. Conditions (a'), (b') and (c') allow to apply \cite[Proposition 3.4]{Graczyk2014}, and hence the solution is non-exploding.  Moreover, conditions on $b_N$ and conditions (c') and (d') imply the non-collision of the solution by \cite[Proposition 4.2]{Graczyk2014}. Note that the continuity of the coefficient functions is not involved in the proofs of \cite[Proposition 3.4 and Proposition 4.2]{Graczyk2014}.

Finally, it is easy to check that conditions (A') - (D') in Section \ref{sec:particles} are satisfied with $\varphi(x) = \ln (1+x^2)$, and  the conclusion follows from Theorem \ref{thm3}.
\end{proof}

The following result is a direct consequence of Corollary \ref{Coro-4.1} and Theorem \ref{thm4}.

\begin{corollary} \label{Coro-4.2}
For the system of SDEs \eqref{SDE-particle}, assume that all the conditions in Corollary \ref{Coro-4.1} hold. Besides, suppose there exist continuous functions $b(x)$ and $H(x,y)$, such that $b_N(x)$ converges to $b(x)$ and $N H_N(x,y)$ converges to $H(x,y)$ uniformly as $N$ tends to infinity. If the empirical measure $L_N(0)$ converges weakly as $N$ goes to infinity to a measure $\mu_0$ almost surely, and the sequence $L_N$ has a limit measure $\mu$ in $C([0,T], M_1(\mathbb{R}))$ for a fixed number $T > 0$, then the measure $\mu$ satisfies the equation
\begin{align} 
	\int \dfrac{\mu_t(dx)}{z-x}
	= \int \dfrac{\mu_0(dx)}{z-x} + \int_{0}^t \left[ \int \dfrac{b(x)}{(z-x)^2} \mu_s(dx) \right] ds + \int_{0}^t \left[ \iint \dfrac{H(x,y)}{(z-x) (z-y)^2} \mu_s(dx) \mu_s(dy) \right] ds,
\end{align}
for $z \in \mathbb{C} \setminus \mathbb{R}$, $t \in [0,T]$.
\end{corollary}

}




\section{Discussion on the equation \eqref{limit point equation 4.3.25} of the limit measure}\label{sec: discussion}

Consider the equation \eqref{limit point equation 4.3.25}  of limit measure
\begin{align*}
	\int \dfrac{\mu_t(dx)}{z-x}
	=& \int \dfrac{\mu_0(dx)}{z-x} + \int_{0}^t \left[ \int \dfrac{b(x)}{(z-x)^2} \mu_s(dx) \right] ds\notag \\
	&+ \int_{0}^t \left[ \iint \dfrac{G(x,y)}{(z-x) (z-y)^2} \mu_s(dx) \mu_s(dy) \right] ds, ~~ \text{ for } z \in \mathbb{C} \setminus \mathbb{R}.
\end{align*}
The uniqueness of the limit measure $\mu_t(dx)$ is obtained so far  only for some special cases in Section \ref{sec:particles} by solving  \eqref{limit point equation 4.3.25} directly with the help of the scaling property \eqref{scaling property for Wishart}. For general cases, the uniqueness is still unknown.  

 In this section, we further explore equation \eqref{limit point
   equation 4.3.25} assuming self-similarity on the eigenvalues
 $\lambda_i^N(t)$, which hopefully may shed some light on solving the
 issue of the uniqueness of  the limit measure.

Recalling that $G(x,y)$ is the limit of $N G_N(x,y)$ where $G_N(x,y)$ takes the form of \eqref{eq-Gn},  we  assume that $G(x,y) = g^2(x) h^2(y) + g^2(y) h^2(x)$, and then \eqref{limit point equation 4.3.25} becomes
\begin{align} \label{limit point differential equation}
	\partial_t \int \dfrac{\mu_t(dx)}{z-x}
	=& \int \dfrac{b(x)}{(z-x)^2} \mu_t(dx) + \int \dfrac{g^2(x)}{z-x} \mu_t(dx) \int \dfrac{h^2(x)}{(z-x)^2} \mu_t(dx) \nonumber \\
	&\quad+ \int \dfrac{h^2(x)}{z-x} \mu_t(dx) \int \dfrac{g^2(x)}{(z-x)^2} \mu_t(dx).
\end{align}

Suppose that the self-similarity $\lambda_i^N(t) \overset d= t^{\alpha} \lambda_i^N(1)$ holds for some constant $\alpha$, then for any $\varphi \in C_b(\mathbb{R})$
\begin{align}\label{eq-5.2}
	&\int  \varphi(x) \mu_t(dx)
	= \lim_{N_j \rightarrow \infty} \dfrac{1}{N_j} \sum_{i=1}^{N_j} \varphi(\lambda_i^{N_j}(t))\notag\\
	=& \lim_{N_j \rightarrow \infty} \dfrac{1}{N_j} \sum_{i=1}^{N_j} \varphi(t^{\alpha} \lambda_i^{N_j}(1))
	= \int  \varphi(t^{\alpha} x) \mu_1(dx).
\end{align}
Hence,  applying \eqref{eq-5.2}  to $\varphi(x) = (z-x)^{-1}$ and  $\varphi(x) = x(z-x)^{-2}$, we have
\begin{align*}
	&\partial_t \int \dfrac{\mu_t(dx)}{z-x}
	= \partial_t \int \dfrac{\mu_1(dx)}{z - t^{\alpha} x} 	= \int \dfrac{\alpha t^{\alpha-1} x}{(z - t^{\alpha} x)^2} \mu_1(dx) \\
	&= \dfrac{\alpha}{t} \int \dfrac{t^{\alpha} x}{(z - t^{\alpha} x)^2} \mu_1(dx) = \dfrac{\alpha}{t} \int \dfrac{x}{(z-x)^2} \mu_t(dx) \\
	&= - \dfrac{\alpha}{t} \partial_z \int \dfrac{x}{z-x} \mu_t(dx).
\end{align*}
Furthermore, we also have
\begin{align*}
	&\quad \int \dfrac{g^2(x)}{z-x} \mu_t(dx) \int \dfrac{h^2(x)}{(z-x)^2} \mu_t(dx) + \int \dfrac{h^2(x)}{z-x} \mu_t(dx) \int \dfrac{g^2(x)}{(z-x)^2} \mu_t(dx) \\
	&= - \int \dfrac{g^2(x)}{z-x} \mu_t(dx) \partial_z \int \dfrac{h^2(x)}{z-x} \mu_t(dx) - \int \dfrac{h^2(x)}{z-x} \mu_t(dx) \partial_z \int \dfrac{g^2(x)}{z-x} \mu_t(dx) \\
	&= - \partial_z \left[ \int \dfrac{g^2(x)}{z-x} \mu_t(dx) \int \dfrac{h^2(x)}{z-x} \mu_t(dx) \right],
\end{align*}
and 
\begin{align*}
	\int \dfrac{b(x)}{(z-x)^2} \mu_t(dx)
	= - \partial_z \int \dfrac{b(x)}{z-x} \mu_t(dx).
\end{align*}
Thus, \eqref{limit point differential equation} can be simplified as
\begin{align*}
 \dfrac{\alpha}{t} \int \dfrac{x}{z-x} \mu_t(dx)
= \int \dfrac{b(x)}{z-x} \mu_t(dx) + \int \dfrac{g^2(x)}{z-x} \mu_t(dx) \int \dfrac{h^2(x)}{z-x} \mu_t(dx) + C(t),
\end{align*}
where $C(t)$ is a complex constant independent of $z$. Let $|z|
\rightarrow \infty$.  By dominated convergence theorem, we can see that $C(t) \equiv 0$. 

Thus, for $G(x,y) = g^2(x) h^2(y) + g^2(y) h^2(x)$, assuming self-similarity on $\lambda_i^N(t)$, the equation \eqref{limit point equation 4.3.25}  for limit measure $\mu_t(dx)$ becomes
\begin{align} \label{measure equation complex case}
	\dfrac{\alpha}{t} \int \dfrac{x}{z-x} \mu_t(dx)
	= \int \dfrac{b(x)}{z-x} \mu_t(dx) + \int \dfrac{g^2(x)}{z-x} \mu_t(dx) \int \dfrac{h^2(x)}{z-x} \mu_t(dx).
\end{align}

In particular, when $b(x), g^2(x)$ and $h^2(x)$ are polynomial functions (consider, for example, Bru's Wishart process, $\beta$-Wishart process, and Dyson Brownian motion), the above equation can be simplified to a polynomial equation only involving the variable $z$ and the Stieltjes transform $\int \frac{1}{z-x}\mu_t(dx)$ of the limit measure $\mu_t(dx)$. 
 
We also would like to point out that  equation \eqref{measure equation complex case} can be represented via the Hilbert transform, in light of the following lemma (see, e.g., Section 3.1 in \cite{stein2011functional}). 
\begin{lemma}\label{lem-5.1} For $\varphi \in L^2(\mathbb{R})$, in the $L^2(\mathbb{R})$-norm we have
\begin{align*}
	\lim_{v \rightarrow 0^+} \int \dfrac{\varphi(x)}{z-x} dx = - 2 \pi i P(\varphi)(u),
\end{align*}
where $z = u + iv$, and the projective operator $P = (I + i H) / 2$ with $H$ being the Hilbert transform operator.
\end{lemma}

Assume that $\mu_t(dx)=p_t(x)dx$ is absolutely continuous with respect
to the Lebesgue measure.  Applying Lemma \ref{lem-5.1} to \eqref{measure equation complex case}, we have the following equation for the density function $p_t(x)$,
\begin{align*}
	\dfrac{\alpha}{t} (I + i H)(x p_t(x))
	= (I + i H)(b(x) p_t(x)) - \pi i (I + i H)(g^2(x) p_t(x)) (I + i H)(h^2(x) p_t(x)).
\end{align*}
The imaginary part of the equation is
\begin{align*}
	H \left( \left(\dfrac{\alpha}{t} x-b(x) \right) p_t(x) \right)
	= -\pi g^2(x) h^2(x) p_t^2(x) + \pi H(g^2(x) p_t(x)) H(h^2(x) p_t(x)),
\end{align*}
which is equivalent to the real part, noting that $H^2=-I$,
\begin{align*}
	\left(\dfrac{\alpha}{t} x-b(x) \right) p_t(x)
	= \pi g^2(x) p_t(x) H(h^2(x) p_t(x)) + \pi h^2(x) p_t(x) H(g^2(x) p_t(x)).
\end{align*}

\section*{Acknowledgment}

{\red
  The authors are grateful to the editor and a referee for numerous {valuable and detailed} comments.  Addressing these comments {has} led to {a significant improvement}  of the paper.
}


\bibliography{WishartProcess}

\end{document}